\documentclass[12pt,article]{article}
\usepackage[margin=1in]{geometry}
\usepackage{amsmath,amsthm,amssymb}
\usepackage{IEEEtrantools}
\usepackage{epsfig}
\usepackage[]{graphicx}
\usepackage{epstopdf}
\usepackage[colorlinks]{hyperref}
\usepackage{listings}
\usepackage{authblk}
\DeclareMathOperator{\e}{e}
\DeclareMathOperator{\diam}{diam}

\DeclareMathOperator{\supp}{supp}

\newtheorem{theorem}{Theorem}[section]
\newtheorem{lemma}[theorem]{Lemma}
\newtheorem{remark}[theorem]{Remark}

\newtheorem{definition}[theorem]{Definition}
\newtheorem{problem}[theorem]{Problem}

\title{Sharp Fourier decay estimates for measures supported on the well-approximable numbers}
\author[1]{Robert Fraser}
\author[2]{Thanh Nguyen}
\affil[1]{Wichita State University}
\affil[2]{Indiana University}
\date{September 2024}

\begin{document}

\maketitle

\begin{abstract}
We construct a measure on the well-approximable numbers whose Fourier transform decays at a nearly optimal rate. This gives a logarithmic improvement on a previous construction of Kaufman.
\end{abstract}

\section{Introduction and Background}

\subsection{Harmonic analysis on fractal sets}
An interesting class of problems in harmonic analysis involves determining information about the Fourier transform of a compactly supported measure $\mu$ given information about the support $\supp \mu$ of the measure $\mu$. A standard result in this area is Frostman's lemma, which states that if $E$ is a set of Hausdorff dimension $s$, then for any $t < s$, there exists a Borel probability measure $\mu_t$ supported on $E$ satisfying the condition that
\begin{equation}\label{eq:l2averageFourier}
\int_{\xi \in \mathbb{R}^n} |\hat \mu_t(\xi)|^2 (1 + |\xi|)^{-t} < \infty.
\end{equation}
Frostman's lemma states that, up to an $\epsilon$-loss in the exponent, the set $E$ supports a measure whose Fourier transform decays like $|\xi|^{-s/2}$ in an $L^2$-average sense.

This version of Frostman's lemma motivates the definition of Fourier dimension. The \text{Fourier dimension} of a set $E \subset \mathbb{R}^n$ is the supremum of those values $0 \leq s \leq n$ such that $E$ supports a Borel probability measure $\mu_s$ satisfying the pointwise condition
\begin{equation}\label{eq:pointwiseFourier}
|\hat \mu_s(\xi)| \lesssim (1 + |\xi|)^{-s/2}.
\end{equation}
Observe that the condition \eqref{eq:pointwiseFourier} for some value of $s$ implies equation \eqref{eq:l2averageFourier} for any $t < s$. However, there is no reason to expect a converse statement to hold; in fact, if $E$ is the usual middle-thirds Cantor set, there is no Borel probability measure $\mu$ on $E$ such that $|\hat \mu(\xi)| \to 0$ as $|\xi| \to \infty$. A measure $\mu$ such that $|\hat \mu(\xi)| \to 0$ as $\xi \to \infty$ is called a \textbf{Rajchman measure}. On the opposite extreme, there are a number of examples of sets $E$ of Hausdorff dimension $s$ supporting Borel probability measures satisfying \eqref{eq:pointwiseFourier} for all $t < s$. Such sets are called \textbf{Salem sets}.

If $s = n - 1$, a simple stationary phase calculation shows that the usual surface measure on the sphere satisfies the condition
\begin{equation}\label{eq:sphereFourier}
|\hat \mu(\xi)| \leq (1 + |\xi|)^{-(n-1)/2}.
\end{equation}
This well-known computation can be found in the textbooks of Wolff \cite{Wolff03} and Mattila \cite{Mattila15}.
If $n = 1$ and $0 < s < 1$, the first examples of Salem sets were given by Salem \cite{Salem51} via a random Cantor set construction. A later random construction was given by Kahane \cite{Kahane66}, who shows that if $\Gamma : [0,1] \to \mathbb{R}^n $ is a Brownian motion and $E \subset [0,1]$ is a set of Hausdorff dimension $s$, then $\Gamma(E)$ will almost surely have Fourier dimension equal to $2s$. Kahane \cite{Kahane66B} also constructed Salem sets using random Fourier series whose coefficients are given by Gaussian random variables.

The first explicit, deterministic example of a Salem set of fractional dimension in $\mathbb{R}$ was given by Kaufman \cite{Kaufman81}. For an exponent $\tau$, the well-approximable numbers $E(q^{-\tau})$ are defined by
\[E(q^{-\tau}) = \left\{x : \left| x - \frac{p}{q}  \right| \leq q^{-\tau} \, \text{for infinitely many pairs of integers $(p,q)$}\right\}.\]
A classical result of Jarník \cite{Jarnik29} and Besicovitch \cite{Besicovitch28} states that the Hausdorff dimension of $E(q^{-\tau})$ is equal to $\frac{2}{\tau}$. Kaufman shows that $E(q^{-\tau})$ supports a Borel probability measure $\mu$ satisfying
\[\left|\hat \mu(\xi)\right| \lesssim (1 + |\xi|)^{-1/\tau} o(\log |\xi|).\]
Bluhm \cite{Bluhm96} provides an exposition of Kaufman's argument to prove a slightly weaker result in which the $o(\log |\xi|)$ term is replaced by $O(\log |\xi|)$.
More generally, given a function $\psi : \mathbb{N} \to [0, \infty) $, it is of interest to consider the set of $\psi$-approximable numbers
\[E(\psi) = \left\{x : \left| x - \frac{p}{q}  \right| \leq \psi(q) \, \text{for infinitely many pairs of integers $(p,q)$}\right\}.\]
Hambrook \cite{Hambrook14} obtains lower bounds on the Fourier dimension of such sets in terms of the function $\psi$.
\subsection{Some problems in geometric measure theory}
In this paper, we will consider the question of locating sets $E$ satisfying more precise estimates than \eqref{eq:pointwiseFourier} under the constraint that $E$ has finite Hausdorff measure. As a motivating example, consider the $(n-1)$-dimensional sphere in $\mathbb{R}^n$. This set has positive and finite $(n-1)$-dimensional Hausdorff measure and supports a measure $\mu$ with Fourier transform satisfying \eqref{eq:sphereFourier}. Mitsis \cite{Mitsis02} posed the following problem.

\begin{problem}[Mitsis's problem]\label{prob:mitsis}
For which values of $0 < s < n$ does there exist a measure $\mu$ such that $\mu$ simultaneously satisfies the ball condition
\[\mu(B(x,r)) \sim r^{s} \quad \text{for all $x \in \supp \mu$ and all $r > 0$}\]
and the Fourier decay condition
\[|\hat \mu(\xi)| \leq |\xi|^{-s/2}?\]
\end{problem}

We will consider a related problem. Let $0 \leq s \leq n$. Recall that a subset $E$ of $\mathbb{R}^n$ is said to be an \textbf{$s$-set} if the Hausdorff measure $\mathcal{H}^s(E)$ satisfies $0 < \mathcal{H}^s(E) < \infty$.
\begin{problem}[Fourier transform on $s$-sets]\label{prob:ssetFourier}
For which values of $0 < s < n$ does there exist an $s$-set $E$ supporting a measure $\mu$ such that $\mu$ satisfies the Fourier decay condition
\[|\hat \mu(\xi)| \leq |\xi|^{-s/2}?\]
\end{problem}
Of course, such a set $E$ must be a Salem set of Hausdorff dimension $s$.

This problem can be extended to a question about generalized Hausdorff dimension. Recall that a positive, increasing function $\alpha$ is said to be a \textbf{dimension function} if $\alpha(u) \to 0$ as $u \to 0$. We will say that $E$ is an $\alpha$-set if $0 < \mathcal{H_{\alpha}(E)} < \infty$, where $\mathcal{H}_{\alpha}$ is the generalized Hausdorff measure associated to $\alpha$. The following question generalizes the previous one:
\begin{problem}[Fourier transform on $\alpha$-sets]\label{prob:alphasetFourier}
For which dimension functions $\alpha$ does there exist an $\alpha$-set $E$ supporting a measure $\mu$ such that $\mu$ satisfies the Fourier decay condition
\[|\hat \mu(\xi)| \lesssim \sqrt{\alpha(1/\xi)} \quad \text{for $|\xi| \geq 1$?}\]
\end{problem}

We conjecture that the only such dimension functions $\alpha$ are integer powers $\alpha(u) = u^{-s}$ for integers $0 \leq s \leq n$.

On the other hand, we also wish to pose the problem of determining the optimal Fourier decay estimates for measures supported on the set of well-approximable numbers $E(\psi)$.
\begin{problem}[Fourier decay of measures supported on $E(\psi)$]\label{prob:epsiFourier}
Fix a function $\psi$. For which functions $\Theta$ does there exist a measure $\mu$ supported on $E(\psi)$ such that
\[|\hat \mu(\xi)| \lesssim \Theta(\xi)?\]
\end{problem}
Although we are unable to answer Problems \ref{prob:ssetFourier}, \ref{prob:alphasetFourier}, and \ref{prob:epsiFourier} in this work, we are able to obtain ``near"-answers to all three of these questions if the dimension function $\alpha$ or the approximation function $\psi$ decay at a polynomial rate.
\subsection{Notation}
In this paper, constants are always allowed to depend on the parameters $\tau, \sigma, $ and $\rho$. Any dependence on these parameters will always be suppressed for simplicity of notation.

If $A$ and $B$ are any two quantities, we write $A = O(B)$ or $A \lesssim B$ to imply that $A \leq C B$ for some constant $C$ that does not depend on $A$ or $B$ (but may depend on $\tau, \sigma$, or $\rho$). We write $B \gtrsim A$ to mean the same thing as $A \lesssim B$. If $A \lesssim B$ and $B \lesssim A$, we write $A \sim B$. If the implicit constant in any of these inequalities is allowed to depend on some other parameter such as $\epsilon$, we write $A \lesssim_{\epsilon} B$, $A \gtrsim_{\epsilon} B$, or $A \sim_{\epsilon} B$. 

If $A(x)$ and $B(x)$ are functions of a variable $x$, we write $A(x) \lessapprox B(x)$ if $A(x) \lesssim_{\epsilon}x^{\epsilon} B(x)$ for every $\epsilon > 0$. So, for example, we write 
\[x^3 \exp (\sqrt{\log x}) \log x \log \log x \lessapprox x^3.\]
If $A(x) \lessapprox B(x)$ and $B(x) \lessapprox A(x)$, we write $A(x) \approx B(x)$.
\section{Results}
First, we describe a result in the direction of Problem \ref{prob:epsiFourier}.
\begin{theorem} \label{thm:slowdecay}
Let $\psi(q)$ be an arbitrary nonnegative, decreasing function satisfying the conditions
\begin{equation}\label{eq:taudef}
2 < \lim_{q \to \infty} -\frac{\log(\psi(q))}{\log q} = \tau < \infty.
\end{equation}
Suppose also that there exists $\sigma > 1$ such that $\psi$ satisfies the polynomial-type decay condition
\begin{equation}\label{eq:sigmadef}
\frac{\psi(q_1)}{\psi(q_2)} \geq \left(\frac{q_2}{q_1} \right)^{\sigma} \quad \text{for $q_2 > q_1$ sufficiently large.}
\end{equation}
Suppose further that $1 \leq \chi(q) \leq \log q$ is a nonnegative function that satisfies 
\begin{equation}\label{eq:chicondition}
\sum_{\substack{q=1 \\ q \text{ prime}}}^{\infty} \frac{1}{q\chi(q)} = \infty.
\end{equation}
and also satisfies the subpolynomial-type growth condition for any $\epsilon > 0$:
\begin{equation}\label{eq:chigrowth}
\frac{\chi(q_2)}{\chi(q_1)} < \left(\frac{q_2}{q_1}\right)^{\epsilon} \quad \text{for $q_1, q_2$ sufficiently large depending on $\epsilon$.}
\end{equation}

Then for any increasing function $\omega$ with $\lim_{\xi \to \infty} \omega(\xi) = \infty$, there exists a Borel probability measure $\mu$ supported on a compact subset of the $\psi$-well-approximable numbers satisfying the estimate
\begin{equation}\label{eq:slowdecay}
|\hat \mu_{\chi, \omega}(\xi)| \lesssim \frac{\omega(|\xi|)}{\psi^{-1}(1/|\xi) \chi (\psi^{-1}(1 /|\xi|))} \quad \text{for all $\xi \in \mathbb{R}$}.
\end{equation}
\end{theorem}
In order to simplify our notation, we define
\begin{equation}\label{eq:thetadef}
\theta(\xi) := \frac{1}{\psi^{-1}(1/\xi) \chi(\psi^{-1}(1/\xi))}.
\end{equation}
\begin{remark}\label{rmk:slowdecayex}
If $\psi(q) = q^{-\tau}$, Theorem \ref{thm:slowdecay} gives estimates that improve on those of Kaufman \cite{Kaufman81}. In this case, the estimate \eqref{eq:slowdecay} becomes
\[|\hat \mu_{\chi, \omega}(\xi)| \lesssim |\xi|^{-1/\tau} \frac{\omega(|\xi|)}{\chi(|\xi|^{1/\tau})}.\]
Observe that, for example, the choice $\chi(q) = \log \log q$ satisfies \eqref{eq:chicondition}. On the other hand, $\omega$ can be taken to be any function that increases to $\infty$, so it is possible to choose $\omega(\xi) = \log \log \log \xi$, for example. Hence there exists a measure $\mu$ supported on the well-approximable numbers satisfying
\[|\hat \mu(\xi)| \lesssim |\xi|^{-1/\tau} \frac{\log \log \log |\xi|}{\log \log |\xi|} \ll |\xi|^{-1/\tau}.\]

\end{remark}

Our next result is in the direction of Problem \ref{prob:alphasetFourier}.
\begin{theorem} \label{thm:compact}
Let $\alpha$ be a dimension function with
\begin{equation}\label{eq:alphalimit}
0 < \lim_{x \to 0}\frac{\log \alpha (x)}{\log x} = \nu < \infty
\end{equation}
and for some $\rho < 1$ such that 
\begin{equation}\label{eq:alphagrowth}
\frac{\alpha(x_1)}{\alpha(x_2)} \geq \left(\frac{x_1}{x_2} \right)^{\rho}
\end{equation}
for sufficiently small $x_1 < x_2$.

Let $\omega$ be an increasing function such that $\lim_{\xi \to \infty} \omega(\xi) = \infty$. Then there exists a compact set $F_\alpha$ of zero $\alpha$-Hausdorff measure such that there exists a measure $\mu_{\alpha, \omega}$ supported on $F_\alpha$ satisfying
\[\left|\hat \mu(\xi)\right|\lesssim \sqrt{\alpha (1/|\xi|)}\omega (|\xi|)\]
for all $s \neq 0$. Such a set is given by an appropriately chosen subset of the well-approximable numbers $E(\psi)$ where
\begin{equation} \label{eq:alpha def}
\psi(q) = \alpha^{-1}\left(q^{-2}\right).    
\end{equation}
\end{theorem}
\begin{remark}
Although this does not provide an answer to Problem \ref{prob:alphasetFourier}, it comes within an arbitrarily slowly growing function of answering this problem. In other words, any improvement on the estimate of Theorem \ref{thm:compact} will give an answer to Problem \ref{prob:alphasetFourier}. 
\end{remark}
\begin{remark}\label{rmk:compacthmconditions}
Observe that the condition \eqref{eq:alphalimit} on $\alpha$ implies the condition \eqref{eq:taudef} on $\psi$ for $\tau = 2/\nu$. A simple calculation also shows that the condition \eqref{eq:alphagrowth} implies the condition \eqref{eq:sigmadef} with $\sigma = 2/\rho$. This is the only way in which the assumptions \eqref{eq:alphalimit} and \eqref{eq:alphagrowth} will be used.
\end{remark}
Finally, we show that, for any decreasing approximation function $\psi$, the set $E(\psi)$ supports a Rajchman measure. This improves a result of Bluhm \cite{Bluhm00} constructing a Rajchman measure supported on the set of Liouville numbers.
\begin{theorem} \label{thm:fastdecay}
For an arbitrary nonnegative, decreasing function $\psi$ there exists a Rajchman measure, $\mu$, supported on a compact subset of the $\psi$-well-approximable numbers.
\end{theorem}
In a recent work, Polasek and Rela \cite{PolasekRela24} improve Bluhm's result in a different way by showing an explicit Fourier decay bound on the set of Liouville numbers. They show that if $f : \mathbb{R}^+ \to \mathbb{R}^+$ is any function such that 
\[\limsup_{\xi \to \infty} \frac{\xi^{-\alpha}}{f(\xi)} = 0 \quad \text{for all $\alpha > 0$},\]
then there exists a measure $\mu_f$ supported on the set of Liouville numbers such that $ |\hat \mu_f(\xi)| \lesssim f(|\xi|)$ for all $\xi$; on the other hand, if $g : \mathbb{R}^+ \to \mathbb{R}^+$ is any function such that 
\[\liminf_{\xi \to \infty} \frac{\xi^{-\alpha}}{f(\xi)} > 0 \quad \text{for some $\alpha > 0$},\]
then there does not exist a measure $\mu_g$ supported on the set of Liouville numbers such that $|\hat \mu_g(\xi)| \lesssim g(|\xi|)$ for all $\xi \in \mathbb{R}$.
\section{Convolution stability lemmas} 
The proofs of the main results of this paper rely on the construction of a sequence of functions which will approximate the measures that satisfy the statements of the theorem. The functions of the sequence are themselves a product of functions. In the frequency space, these products become convolutions and a major component of the proof is show that the Fourier decay estimates of these functions remain stable as the number of convolutions tends to infinity. The following two lemmas will be referred to when making an argument for stability by induction. This first lemma will be applied to Theorem \ref{thm:slowdecay} and Theorem \ref{thm:compact}.
\begin{lemma}{\textnormal{(Convolution Stability Lemma)}}\label{lem:convolutionstability}
Let $\omega : \mathbb{N} \to \mathbb{R}^+$ be a function that increases to infinity such that $\omega(t) \leq \log t$ for $t \geq 2$. Suppose that $N_2$ is sufficiently large depending on $\omega$ and $N_1$. Moreover, let $G,H : \mathbb{Z} \to \mathbb{C}$ be functions satisfying the following bounds for some $N_3 > \frac{1}{\psi (\beta(N_2))^2}$:

\begin{IEEEeqnarray} {rCll}
|G(s)| & \leq & 1 & \quad \text{for all $s \in \mathbb{Z}$} \label{eq:G global}\\
G(0) & = &1& \label{eq:G(0)}\\
G(s) & = & 0 & \quad 0 < |s| \leq N_2 \label{eq:GSmall}\\
|G(s)| & \lesssim & \theta(|s|) & \quad \text{everywhere} \label{eq:GEverywhere}\\ 
|G(s)| & \lesssim & \exp\left(-\frac{1}{2}\left|\frac{s}{2N_3}\right|^{\frac{\sigma+1}{4 \sigma}}\right) &\quad \text{when } |s|\geq2N_3 \label{eq:GVeryLarge}\\
|H| & \leq & 2 & \label{eq:H global}\\ 
|H(s)| &\lesssim &  \theta(|s|) \omega(|s|) & \quad \text{everywhere} \label{eq:HEverywhere}\\ 
|H(s)| & \lesssim & \exp\left(-\frac{1}{2}\left|\frac{s}{8N_1}\right|^{\frac{\sigma+1}{4 \sigma}}\right) & \quad \text{when } |s|\geq8N_1. \label{eq:HVeryLarge}
\end{IEEEeqnarray}

Then

\begin{IEEEeqnarray} {rCll}
|H*G(s) - H(s)| & \lesssim & N_2^{-99} & \quad \text{when } 0\leq|s| < N_2/4 \label{eq:convolve1}\\ 
|H*G(s)| & \lesssim & \theta(|s|) \omega(|s|) & \quad \text{everywhere} \label{eq:convolve2}\\ 
|H*G(s)| & \lesssim & \exp\left(-\frac{1}{2}\left|\frac{s}{8N_3}\right|^{\frac{\sigma+1}{4 \sigma}}\right) & \quad \text{when } |s| \geq 8N_3. \label{eq:convolve3}
\end{IEEEeqnarray}

\end{lemma}

A different version of this lemma will be applied to prove Theorem \ref{thm:fastdecay}.
\begin{lemma}[Convolution Stability Lemma 2] \label{lem:convolutionstability2}
Let $\delta < \frac{1}{N_1^2}$. Suppose that $N_2$ is sufficiently large depending on $\omega$ and $N_1$. Moreover, let $G,H : \mathbb{Z} \to \mathbb{C}$ be functions satisfying the following bounds for some $N_3 > \frac{1}{\psi (\beta(N_2))}$:

\begin{IEEEeqnarray} {rCll}
|G(s)| & \leq & 1 \quad & \text{for all $s \in \mathbb{Z}$} \label{eq:2 G global}\\
G(0) & = &1& \label{eq:2 G(0)}\\
G(s) & = & 0 & \quad 0 < |s| \leq N_2 \label{eq:GSmall2}\\
|G(s)| & \lesssim & \delta & \quad s \neq 0 \label{eq:2 GEverywhere}\\ 
|G(s)| & \lesssim & \exp\left(-\frac{1}{2}\left|\frac{s}{2N_3}\right|^{\frac{3}{4}}\right) &\quad \text{when } |s|\geq2N_3 \label{eq:2 GVeryLarge}\\
|H| & \leq & 2 & \label{eq:2 H global}\\ 
|H(s)| & \lesssim & \exp\left(-\frac{1}{2}\left|\frac{s}{8N_1}\right|^{\frac{3}{4}}\right) & \quad \text{when } |s|\geq8N_1. \label{eq:2 HVeryLarge}
\end{IEEEeqnarray}

Then

\begin{IEEEeqnarray} {rCll}
|H*G(s) - H(s)| & \lesssim & N_2^{-99} & \quad \text{when } 0 \leq|s| < N_2/4 \label{eq:2 convolve1}\\ 
|H*G(s)| & \lesssim & \delta^{1/2} & \quad \text{when $|s| \geq N_2/4$} \label{eq:2 convolve2}\\ 
|H*G(s)| & \lesssim & \exp\left(-\frac{1}{2}\left|\frac{s}{8N_3}\right|^{\frac{3}{4}}\right) & \quad \text{when } |s| \geq 8N_3. \label{eq:2 convolve3}
\end{IEEEeqnarray}
\end{lemma}
Before proving these lemmas, we need a preliminary estimate on $\theta$. We will show that the function $\theta(\xi)$ decays like $\xi^{-1/\tau}$ up to an $\epsilon$-loss in the exponent.
\begin{lemma}
Let $\psi, \chi$ be as in Theorem \ref{thm:slowdecay}, and let $\theta(\xi)$ be as in \eqref{eq:thetadef}. Then $\theta(|\xi|) \approx |\xi|^{-\frac{1}{\tau}}$ for large $|\xi|$. 
\end{lemma}
\begin{proof}
Since $\psi(q) \approx q^{-\tau}$ by assumption, we have that $\psi^{-1}(t) \approx t^{-1/\tau}$. A similar argument shows that $\chi(t) \approx 1$. Hence $\chi(\psi^{-1}(1/|\xi|)) \approx 1$. Thus
\begin{equation}\label{eq:thetatau}
\frac{1}{\psi^{-1}(1/|\xi|) \chi(\psi^{-1}(1/|\xi|))} \approx |\xi|^{-1/\tau}.
\end{equation}
\end{proof}
\subsection{Proof of Lemma \ref{lem:convolutionstability}}
\begin{proof} First, we prove \eqref{eq:convolve1}. Assume that $0 \leq |s| \leq N_2/4$. Rewrite the expression as
\begin{IEEEeqnarray*} {rCl}
|H * G (s) - H (s)| & = & \left|\sum_{t\in \mathbb{Z}} H(s-t)G (t) - H (s)\right|\\
&=&\left|H(s)G(0) -H(s) +\sum_{t \neq 0} H(s-t)G (t)\right| \\
&\leq & \sum_{t \neq 0} |H(s-t)G (t)|.
\end{IEEEeqnarray*}

Observe that we need only consider summands such that $|t| \geq N_2$ because $G(t) = 0$ for $|t| < N_2$. The previous expression becomes
\[\sum_{|t| \geq N_2} |H (s-t) G (t)|.\]

Apply the bound $\eqref{eq:G global}$ to $|G(t)|$. Notice that $|s-t| \geq |t|/2 \gg N_1$ when $|s| < N_2/4$. We may apply $\eqref{eq:HVeryLarge}$ with $|t|/2$ in place of $s$ to get an upper bound given that the bounding function is decreasing. Hence
\begin{IEEEeqnarray*} {rCl}
\sum_{|t| \geq N_2} |H (s-t) G (t)| &\lesssim & \sum_{|t| \geq N_2}\exp\left(-\frac{1}{2}\left|\frac{t}{8N_1}\right|^{\frac{\sigma+1}{4 \sigma}}\right)\\
& \leq & N_2^{-99}.
\end{IEEEeqnarray*}

The last inequality holds provided that $N_2$ is sufficiently large depending on $N_1$. The next task is to prove the estimate \eqref{eq:convolve2}. For $0 < |s| < N_2/4$, the estimate follows from \eqref{eq:HEverywhere} \eqref{eq:convolve1}. Indeed, the difference
\[|H*G(s) - H(s)| \lesssim N_2^{-99} \lesssim |s|^{-99}.\]
By the estimate \eqref{eq:thetatau}, we have that
\[|s|^{-99} \lesssim \theta(|s|).\]

Now assume that $|s| \geq N_2/4$. We have the inequality
\[|H * G (s)| \leq \mathrm{I} + \mathrm{II},\]
where
\[\mathrm{I} = \sum_{|t| < 2N_1} |H (t)G (s-t)|,\]
and
\[\mathrm{II} = \sum_{|t| \geq 2N_1} |H (t)G (s-t)|.\]\

Beginning with the sum $\mathrm{I}$, we apply \eqref{eq:H global} and observe that $|s-t| \geq |s| -|t| \geq |s|/2$ when $|s| \geq N_2/4$. Then we may apply $\eqref{eq:GEverywhere}$ with $|s|/2$ in place of $|s|$ to get
\begin{IEEEeqnarray*} {rCl}
\mathrm{I} &\lesssim & \theta(|s|) \sum_{|t| < 2N_1} 1\\
& \lesssim & \theta(|s|) \omega(|s|),
\end{IEEEeqnarray*}
provided that $N_2$ is sufficiently large depending on $\omega$ and $N_1$ so that the final inequality holds. To bound the sum $\mathrm{II}$, write

\[\mathrm{II} = \mathrm{A} + \mathrm{B},\]
where 
\[A = \sum_{\substack{|t| \geq 2N_1\\ {|s-t| \leq |s|/2}}} |H (t)G (s-t)|,\]
and 
\[B = \sum_{\substack{|t| \geq 2N_1\\ {|s-t| > |s|/2}}} |H (t)G (s-t)|.\]

To estimate the sum $A$, we apply $\eqref{eq:G global}$ and $\eqref{eq:HVeryLarge}$. Observe that $|s-t| \leq |s|/2$ implies that $|t| \geq |s|/2$. Thus, $t \gg N_1$ when $|s| \geq N_2/4$. Therefore
\[A \lesssim \sum_{|t| \geq |s|/2} \exp\left(-\frac{1}{2}\left|\frac{t}{8N_1}\right|^{\frac{\sigma+1}{4 \sigma}}\right).\]
By the integral test, we get the following upper bound for $\mathrm{A}$:
\[A \lesssim \int_{|s|/2}^\infty \exp\left(-\frac{1}{2}\left|\frac{t}{8N_1}\right|^{\frac{\sigma+1}{4 \sigma}}\right) \text{ } dt.\]
Observe that the integrand is decaying nearly exponentially. From \eqref{eq:thetatau}, we may conclude
\[A \lesssim \theta(|s|).\]

For the sum $B$, we apply \eqref{eq:GEverywhere} to $G$. Additionally, we may apply \eqref{eq:HVeryLarge}. Since $|s|/2 \gg N_1$, we have after making the substitution $u = s - t$ that 
\begin{IEEEeqnarray*} {rCl}
B & \lesssim &  \theta(|s|)  \sum_{|u| \geq |s|/2} \exp\left(-\frac{1}{2}\left|\frac{u}{8N_1}\right|^{\frac{\sigma+1}{4 \sigma}}\right)\\
& \lesssim & \theta(|s|)
\end{IEEEeqnarray*}
where the last inequality is implied by
\[\sum_{|u| \geq |s|/2} \exp\left(-\frac{1}{2}\left|\frac{u}{8N_1}\right|^{\frac{\sigma+1}{4 \sigma}}\right) \leq 1.\]

Combining the bounds for $\mathrm{I}$ and $\mathrm{II}$ completes the proof for \eqref{eq:convolve2}. We turn now to proving \eqref{eq:convolve3}. Assume $|s| \geq 8N_3$. We decompose the convolution as
\[|H * G (s)| \leq \mathrm{I} + \mathrm{II}, \]
where 
\[\mathrm{I} = \sum_{|t| < |s|/2} |H (s-t)G (t)|\]
and
\[\mathrm{II} = \sum_{|t| \geq |s|/2} |H (s-t)G (t)|.\]

Starting with $\mathrm{I}$, we apply \eqref{eq:G global}. Then apply $\eqref{eq:HVeryLarge}$ with $s-t$ in place of $s$, and use the fact that $|s-t|\geq |s|/2$. Then
\[\mathrm{I} \lesssim \sum_{|t| < |s|/2} \exp\left(-\frac{1}{2}\left|\frac{s}{16N_1}\right|^{\frac{\sigma+1}{4 \sigma}}\right).\]
There are at most $|s|/2$ summands in the above sum. Therefore
\[\mathrm{I} \lesssim |s| \exp\left(-\frac{1}{2}\left|\frac{s}{16N_1}\right|^{\frac{\sigma+1}{4 \sigma}}\right).\]
We may bound the prior estimate by a single exponential function by choosing a smaller negative power and eliminating the linear term. Hence,
\[\mathrm{I} \lesssim \exp\left(-\frac{1}{2}\left|\frac{s}{8N_3}\right|^{\frac{\sigma+1}{4 \sigma}}\right).\]

For the sum $\mathrm{II}$, apply the bounds \eqref{eq:GVeryLarge} and \eqref{eq:H global} to get
\begin{IEEEeqnarray*} {rCl}
\mathrm{II} & \lesssim & \sum_{|t| > |s|/2} \exp\left(-\frac{1}{2}\left|\frac{t}{2N_3}\right|^{\frac{\sigma+1}{4 \sigma}}\right).
\end{IEEEeqnarray*}
To bound the above sum, we use the integral test. Thus
\[\mathrm{II} \lesssim \int_{t > |s|/2} \exp \left(-\frac{1}{2} \left| \frac{t}{2 N_3} \right|^{\frac{\sigma + 1}{4 \sigma}} \right)\, dt. \]
To estimate this integral, we begin with a substitution. Let 
\[u = \frac{1}{2}\left|\frac{t}{2N_3}\right|^{\frac{\sigma+1}{4 \sigma}}.\]
Then
\[du = \frac{\sigma +1}{16 \sigma N_3}\left|\frac{t}{2N_3}\right|^{\frac{-3\sigma+1}{4 \sigma}} dt.\]
The integral may be rewritten as
\[\frac{16 \sigma N_3}{\sigma +1}\int_{t = |s|/2}^\infty \exp\left(-u\right) \left(2u\right)^{\frac{3\sigma -1}{\sigma +1}}\text{ } du.\]
Integrating by parts yields
\[\left(\frac{16 \sigma N_3}{\sigma +1}\right) \left(\ -\exp(-u)\left(2u\right)^{\frac{3\sigma -1}{\sigma +1}}\big |_{t = |s|/2}^{\infty} + \frac{6\sigma - 2}{\sigma + 1} \int_{t = |s|/2}^\infty \exp\left(-u\right) \left(2u\right)^{\frac{3\sigma -1}{\sigma +1}-1}\text{ } du\right).\]

It is easy to see that expression above is dominated by the first term and the integral is an error term. Indeed, repeated integration by parts will yield relatively small terms (we are only interesting in finding an estimate up to a multiplicative constant) which contribute a negligible amount to the estimate. We consider only the first term in the estimate and evaluating it at the endpoints to get

\[\mathrm{II} \lesssim N_3 \exp\left(-\frac{1}{2}\left|\frac{s}{4N_3}\right|^{\frac{\sigma+1}{4 \sigma}}\right)\left|\frac{s}{4N_3}\right|^{\frac{3\sigma-1}{4 \sigma}}.\]

Observe that exponential term is dominant for large values of $s$. We may bound the above expression by a single exponential term by choosing a smaller negative power. Hence,

\begin{IEEEeqnarray*} {rCl}
\mathrm{II} & \lesssim & N_3 \exp\left(-\frac{1}{2}\left|\frac{s}{4N_3}\right|^{\frac{\sigma+1}{4 \sigma}}\right)\left|\frac{s}{4N_3}\right|^{\frac{3\sigma-1}{4 \sigma}}\\
& \leq & \exp\left(-\frac{1}{2}\left|\frac{s}{8N_3}\right|^{\frac{\sigma+1}{4 \sigma}}\right).
\end{IEEEeqnarray*}
Combining the estimates on $\mathrm{I}$ and $\mathrm{II}$ completes the proof of the lemma.
\end{proof}
\subsection{Proof of Lemma \ref{lem:convolutionstability2}}

The proof of Lemma \ref{lem:convolutionstability2} shares many similarities with the proof of Lemma \ref{lem:convolutionstability}.
\begin{proof}
Beginning with \eqref{eq:2 convolve1}, assume $|s| \geq N_2/4$ and write
\[|H * G (s) - H (s)| \leq \sum_{|t| \geq N_2} |H (s-t) G (t)|.\]
Apply the estimate \eqref{eq:2 HVeryLarge} with $|t|/2$ in place of $s$ and the estimate \eqref{eq:2 GEverywhere}. Then
\begin{IEEEeqnarray*} {rCl}
|H * G (s) - H (s)| & \lesssim & \delta \sum_{|t| \geq N_2}\exp\left(-\frac{1}{2}\left|\frac{t}{8N_1}\right|^{\frac{3}{4}}\right)\\ 
& \leq & \delta
\end{IEEEeqnarray*}
since the sum may be bounded above by $1$.

In order to prove the estimate \eqref{eq:2 convolve2}, we assume $|s| \geq N_2/4$. Write
\[|H * G (s)| \leq \mathrm{I} + \mathrm{II}, \]
where 
\[\mathrm{I} = \sum_{|t| < N_1} |H (t)G (s-t)|\]
and
\[\mathrm{II} = \sum_{|t| \geq N_1} |H (t)G (s-t)|.\]

For the sum $\mathrm{I}$, apply the estimates \eqref{eq:2 H global} and \eqref{eq:2 GEverywhere}. Then
\begin{IEEEeqnarray*} {rCl}
\mathrm{I} & \lesssim & \delta \sum_{|t| < N_1} 1\\
& \lesssim & \delta N_1\\
& < & \delta^{1/2}.
\end{IEEEeqnarray*}
where the final inequality follows from the fact that $\delta < \frac{1}{N_1^2}$.

For the sum $\mathrm{II}$, consider the term where $s=t$ separately from other summands. Write
\[\mathrm{II} = \sum_{\substack {|t| \geq N_1\\ {s\neq t}}} |H (t)G (s-t)| + |H(s)G(0)|.\]
Apply the estimates \eqref{eq:2 HVeryLarge}, \eqref{eq:2 GEverywhere} and \eqref{eq:2 G(0)}. Then
\begin{IEEEeqnarray*} {rCl}
\mathrm{II} &\lesssim& \delta \sum_{\substack {|t| \geq N_1\\ {s\neq t}}} \exp\left(-\frac{1}{2}\left|\frac{t}{8N_1}\right|^{\frac{3}{4}}\right)  + \exp\left(-\frac{1}{2}\left|\frac{s}{8N_1}\right|^{\frac{3}{4}}\right)\\
& \lesssim & \delta^{1/2}.
\end{IEEEeqnarray*}
The last inequality is implied by the bounds
\[\sum_{\substack {|t| \geq N_1\\ {s\neq t}}} \exp\left(-\frac{1}{2}\left|\frac{t}{8N_1}\right|^{\frac{3}{4}}\right) \lesssim 1\]
and
\[\exp\left(-\frac{1}{2}\left|\frac{s}{8N_1}\right|^{\frac{3}{4}}\right) \lesssim \delta^{1/2}.\]

For the final estimate \eqref{eq:2 convolve3}, assume $|s|\geq 8N_3$ and write
\[|H * G (s)| \leq \mathrm{I} + \mathrm{II}, \]
where 
\[\mathrm{I} = \sum_{|t| < |s|/2} |H (s-t)G (t)|\]
and
\[\mathrm{II} = \sum_{|t| \geq |s|/2} |H (s-t)G (t)|.\]

For the sum $\mathrm{I}$, use the fact that $|s - t| \geq |s|/2$ and apply \eqref{eq:2 G global} and \eqref{eq:2 HVeryLarge} with $|s|/2$. Then
\begin{IEEEeqnarray*} {rCl}
 \mathrm{I} & \lesssim & \sum_{|t| < |s|/2} \exp\left(-\frac{1}{2}\left|\frac{s}{16N_1}\right|^{\frac{3}{4}}\right)\\ 
 & \lesssim & \exp\left(-\frac{1}{2}\left|\frac{s} {8N_3}\right|^{\frac{3}{4}}\right).
\end{IEEEeqnarray*}

For the sum $\mathrm{II}$, apply \eqref{eq:2 H global} and \eqref{eq:2 GVeryLarge} with $|t|$ in place of $s$. Then
\begin{IEEEeqnarray*} {rCl}
 \mathrm{II} & \lesssim & \sum_{|t| > |s|/2} \exp\left(-\frac{1}{2}\left|\frac{t}{2N_3}\right|^{\frac{3}{4}}\right)\\ 
 & \leq & \exp\left(-\frac{1}{2}\left|\frac{s} {8N_3}\right|^{\frac{3}{4}}\right).
\end{IEEEeqnarray*}
This completes the proof of Lemma \ref{lem:convolutionstability2}.
\end{proof}

\section{Doubling functions}
\begin{definition}
If $f : \mathbb{R}^+ \to \mathbb{R}^+$ is a decreasing or eventually decreasing function, we say that $f$ is \textbf{doubling} if $f(\xi/2) \lesssim f(\xi)$ for all sufficiently large $\xi$.
\end{definition}
We will need a few basic facts about doubling functions.
\begin{lemma}
The function $\theta(\xi)$ is doubling.
\end{lemma}
\begin{proof}
The fact that $\theta(\xi) \approx \xi^{-1/\tau}$ implies that $\theta(\xi)$ is eventually decreasing. To see that $\theta(\xi)$ is doubling, note that for sufficiently large $q_1$ and $q_2$ with $q_1 < q_2$, we have the assumption \eqref{eq:sigmadef}, which is reproduced below for convenience.
\[\frac{\psi(q_1)}{\psi(q_2)} \geq \left(\frac{q_2}{q_1} \right)^{\sigma}.\]
Since $\psi^{-1}$ is decreasing, we have that $\psi^{-1}(1/\xi) > \psi^{-1}(2/\xi)$. If $\xi$ is sufficiently large that \eqref{eq:sigmadef} applies with $q_1 = \psi^{-1}(2/\xi)$ and $q_2 = \psi^{-1}(1/\xi)$, then we have
\[\frac{\psi(q_1)}{\psi(q_2)} = \frac{2/\xi}{1/\xi} \geq \left(\frac{\psi^{-1}(1/\xi)}{\psi^{-1}(2/\xi)} \right)^{\sigma}.\]
Hence
\begin{IEEEeqnarray*}{Cl}
& \frac{\theta(\xi/2)}{\theta(\xi)} \\
= & \frac{\psi^{-1}(1/\xi) \chi(\psi^{-1}(1/\xi))}{\psi^{-1}(2/\xi) \chi(\psi^{-1}(2/\xi))} \\
\leq & \left(\frac{\psi^{-1}(1/\xi)}{\psi^{-1}(2/\xi)} \right)^{1 + \epsilon} \\
\leq & 2^{(1 + \epsilon)/\sigma}.
\end{IEEEeqnarray*}
Hence $\theta(\xi)$ is doubling.
\end{proof}
Next, we show under very general conditions that a function with limit $0$ must admit a decreasing, doubling majorant.
\begin{lemma}\label{lem:doublingmajorant}
Suppose that $M: \mathbb{Z} \to \mathbb{C}$ is any function such that $|M(s)| \to 0$ as $|s| \to \infty$. Then there is a decreasing function $N: \mathbb{R}^+ \to \mathbb{R}^+$ such that $N(\xi) \to 0$ as $\xi \to \infty$ satisfying the doubling property such that $|M(s)| \leq N(|s|)$ for all $s \in \mathbb{Z}$.
\end{lemma}
\begin{proof}
First, we replace $M$ by a decreasing function $M_1 : \mathbb{R}^+ \to \mathbb{R}^+$ as follows. For $s \in \mathbb{N}$, define
\[M_1(s) = \sup_{|t| \geq s} |M(t)|.\]
Then $M_1$ is decreasing on $[0,\infty)$, $|M_1(s)| \leq M(|s|)$ for all $s \in \mathbb{Z}$, and $\lim_{s \to \infty} M_1(s) = 0$.

We construct $N$ by taking the average of $M$. For $\xi \in \mathbb{R}^+$, define
\[N(\xi) = \frac{1}{\lfloor \xi \rfloor + 1} \sum_{\substack{t \in \mathbb{N} \\ t \leq \xi}} M_1(t).\]
As $N$ is an average of a decreasing function, it follows that $N$ is decreasing; moreover, since $M_1(t) \to 0$ as $t \to \infty$, it follows that $N(\xi) \to 0$ as $\xi \to \infty$. Furthermore, it is easy to see that $M_1(s) \leq N(s)$ for $s \in \mathbb{N}$:
\begin{IEEEeqnarray*}{rCl}
N(s) & = & \frac{1}{s + 1} \sum_{t=0}^s M_1(t) \\
& \geq & \frac{1}{s + 1} \sum_{t = 0}^s M_1(s) \\
& = & \frac{1}{s + 1} (s + 1) M_1(s) \\
& = & M_1(s),
\end{IEEEeqnarray*}
So $|M(s)| \leq M_1(|s|) \leq N(|s|))$ for all $s \in \mathbb{Z}$. 

It only remains to verify that $N(s)$ has the doubling property \eqref{eq:doubling}. We have for $s \neq 0$ that
{\allowdisplaybreaks \begin{IEEEeqnarray*}{rCl}
N(s/2) & = & \frac{1}{\lfloor s/2 \rfloor + 1} \sum_{\substack{t \leq s/2 \\ t \in \mathbb{N}}} M_1(t) \\
& \leq &   \frac{1}{\lfloor s/2 \rfloor + 1} \sum_{\substack{t \leq \lfloor s/2 \rfloor \\ t \in \mathbb{N}}} M_1(t) + \frac{1}{\lfloor s/2 \rfloor + 1} \sum_{\substack{\lfloor s/2 \rfloor + 1 \leq t \leq 2 \lfloor{s/2} \rfloor  + 1 \\ t \in \mathbb{N}}} M_1(t)  \\
& \leq &   \frac{1}{\lfloor s/2 \rfloor + 1} \sum_{\substack{t \leq \lfloor s/2 \rfloor \\ t \in \mathbb{N}}} M_1(t) + \frac{1}{\lfloor s/2 \rfloor + 1} \sum_{\substack{\lfloor s/2 \rfloor + 1 \leq t \leq s + 1 \\ t \in \mathbb{N}}} M_1(t)  \\
& \leq & \frac{2}{2 \lfloor s/2 \rfloor + 2} \sum_{\substack{t \leq s + 1 \\ t \in \mathbb{N}}} M_1(t). \\
& \leq &\frac{2}{s} \sum_{\substack{t \leq s + 1 \\ t \in \mathbb{N}}} M_1(t) \\
& \leq & \frac{2}{s} \sum_{\substack{t \leq s \\ t \in \mathbb{N}}} M_1(t) + \frac{2}{s} M_1(s + 1) \\
& \leq & \frac{2}{s} \sum_{\substack{t \leq s \\ t \in \mathbb{N}}} M_1(t) + \frac{2}{s} \sum_{\substack{t \leq s \\ t \in \mathbb{N}}}  M_1(t) \\
& \leq & \frac{4}{s} \sum_{\substack{t \leq s \\ t \in \mathbb{N}}} M_1(t) \\
& \leq & \frac{8}{s + 1} \sum_{\substack{t \leq s \\ t \in \mathbb{N}}} M_1(t) \\
& = & 8 N(s),
\end{IEEEeqnarray*}}
as desired.
\end{proof}

\section{Single-factor estimates}\label{sec:singlescale}
\subsection{Single-factor estimates for Theorem \ref{thm:slowdecay} and Theorem \ref{thm:compact}}\label{subsec:singlescaleslowdecay}
In this section, we construct a function $g_k$ with its support contained in intervals centered at rational numbers with denominator close to some number $M_k$. Let $\psi(q)$ be a function satisfying \eqref{eq:taudef} and \eqref{eq:sigmadef}. Suppose $\chi(q)$ is a function satisfying \eqref{eq:chicondition}. In the case of Theorem \ref{thm:compact}, we take $\chi(q) \equiv 1$. 

Let $M_k$ be a large positive integer. We choose an integer $\beta(M_k)$ and a positive real number $C_k$ so that

\[1 \leq \sum_{\substack{M_k \leq q \leq \beta(M_k) \\ q \text{ prime} }} \frac{1}{q \chi(q)} = C_k \leq 2.\]

The support of $g_k$ will be contained in a family of intervals centered at rational numbers whose denominator is a prime number between $M_k$ and $\beta(M_k)$. 

We choose a nonnegative function $\phi \in C_c^\infty$ with support in the interval $[-1/2,0]$ satisfying the conditions
\begin{equation}\label{eq:phinormalization}
\hat \phi (0) = 1
\end{equation}
and
\begin{equation}\label{eq:schwartztail}
 \hat \phi (s) \lesssim \exp \left( -|s|^{\frac{\sigma + 1}{2 \sigma}}\right).
\end{equation}
The existence of such a function is guaranteed by a result of Ingham \cite{Ingham34}.

Let 
\[\phi_{p,q}(x) = \frac{1}{q^2\chi(q) \psi (q)} \phi\left(\frac{1}{\psi(q)}\left(x-\frac{p}{q}\right)\right)\]
Now define
\[g_k (x) = C_k^{-1} \sum_{\substack {M_k \leq q < \beta(M_k) \\ q \text{ prime}}} \sum_{p = 1}^{q} \phi_{p,q}(x).\]
Observe that the function $g_k$ is supported on the interval $[0,1]$.
\begin{lemma} \label{lemma: gkestimate2.1}
Suppose $g_k$ is defined as above. Then we have the following estimates for $s \in \mathbb{Z}$.
\begin{IEEEeqnarray}{rCll} 
\hat g_k(0) & = & 1 \label{eq:gkhat02.1} \\
\hat g_k(s) & = & 0 \quad & \text{if $0 < |s| < M_k$} \label{eq:gkhatssmall2.1} \\
|\hat g_k(s)| & \lesssim & \theta(|s|) \quad & \text{if $s \neq 0$}  \label{eq:gkhateverywhere2.1}\\
|\hat g_k(s)| & \lesssim & \exp \left(-\frac{1}{2} (\psi(\beta(M_k))^{2}|s|)^{\frac{\sigma + 1}{4 \sigma}} \right) \quad & \text{if $|s| \geq \psi(\beta(M_k))^{-2}$.}\label{eq:gkhatslarge2.1}
\end{IEEEeqnarray}
\end{lemma}
\begin{proof}
A simple calculation gives us that

\[\hat g_k (s) = C_k^{-1} \sum_{\substack {M_k \leq q < \beta(M_k) \\ q \text{ prime}}} \frac{1}{q^2 \chi(q)}\sum_{p = 1}^{q} \e\left(\frac{p}{q}s\right) \hat \phi\left(\psi(q) s\right)\]
where $\e(u) = \e^{-2 \pi i u}$. The sum in $p$ has the value
\[\sum_{p=1}^q \e \left(\frac{p}{q} s \right) = \begin{cases}
    q & \text{if $q \mid s$} \\
    0 & \text{if $q \nmid s$.}
\end{cases}\]

Therefore, if $s = 0$, then the above sum will be equal to $1$, and if $0 < |s| < M_k$, then the above sum will vanish. This proves \eqref{eq:gkhat02.1} and \eqref{eq:gkhatssmall2.1}. Thus,

\[\hat g_k (s) = C_k^{-1} \sum_{\substack {M_k \leq q < \beta(M_k) \\ {\substack q \text{ prime}} \\ q \mid s}} \frac{\hat \phi\left(\psi(q) s\right)}{q\chi(q)}.\]

For $|s| \geq M_k$, we split the above sum into three pieces according to the size of $q$. We write

\[\hat g_k (s) = C_k^{-1} (\mathrm{I} + \mathrm{II} + \mathrm{III}),\]
where
\begin{IEEEeqnarray*}{rCl}
\mathrm{I} & = & \sum_{\substack {q \geq \psi^{-1}(1/|s|) \\ {\substack q \text{ prime}} \\ q \vert s}} \frac{\hat \phi\left(\psi(q) s\right)}{q\chi(q)}, \\
\mathrm{II} & = &  \sum_{\substack {\psi^{-1} (1/ \sqrt{|s|}) \leq {q \leq \psi^{-1}(1/|s|)} \\ {\substack q \text{ prime}} \\ q \vert s}} \frac{\hat \phi\left(\psi(q) s\right)}{q\chi(q)}, \\
\mathrm{III} & = &  \sum_{\substack {q < \psi^{-1}(1/\sqrt{|s|}) \\ {\substack q \text{ prime}} \\ q \vert s}} \frac{\hat \phi\left(\psi(q) s\right)}{q\chi(q)}.
\end{IEEEeqnarray*}

\paragraph*{Estimate for $\mathrm{I}$.} For the sum $\mathrm{I}$, we observe that the number of summands is $\lesssim 1$. This observation is a consequence of assumption \eqref{eq:taudef} since it is implied that for a large enough $q$ depending on $\epsilon$ we have,
\[q^{-\tau-\epsilon} \leq \psi(q) \leq q^{-\tau + \epsilon}\]
which gives us
\[t^{-\frac{1}{\tau + \epsilon}}\leq \psi^{-1}(t) \leq t^{- \frac{1}{\tau - \epsilon}}\]
since $\psi$ is decreasing. Taking logarithms, we conclude
\begin{equation} \label{eq:bound on psi inverse}
\frac{1}{\tau + \epsilon} \log |s| \leq \log \psi^{-1} \left(\frac{1}{|s|} \right) \leq  \frac{1}{\tau - \epsilon} \log |s|.    
\end{equation}
Hence, the number of summands in the sum $\mathrm{I}$ is at most $\frac{\log |s|}{\log \psi^{-1}(1/|s|)} \lesssim 1$. 

Apply the bound $\hat \phi\left(\psi(q) s\right) \leq 1$ to each summand to get
\[\sum_{\substack {q \geq \psi^{-1}(1/|s|) \\ {\substack q \text{ prime}} \\ q \vert s}} \frac{\hat \phi\left(\psi(q) s\right)}{q\chi(q)} \lesssim \theta(|s|).\]

\paragraph*{Estimate for $\mathrm{II}$.} For the sum $\mathrm{II}$, we observe that there are $\lesssim 1$ summands by a similar argument as for the sum $\mathrm{I}$. We apply the bound \eqref{eq:schwartztail} to show that the summand is bounded above by
\[\frac{\exp(- |\psi(q) s|^{\frac{\sigma + 1}{2 \sigma}})}{q \chi(q)}\]
If $q = \psi^{-1}(1/|s|)$, then  
\begin{equation}\label{eq:sum2bound}
\frac{\exp(- |\psi(q) s|^{\frac{\sigma + 1}{2 \sigma}})}{q \chi(q)} \lesssim \theta(|s|).
\end{equation}

It is enough to show for each $q < \psi^{-1}(1/|s|)$ that 

\begin{equation}\label{eq:summandincreasing}
\frac{\exp (- | \psi(q + 1) s|^{\frac{\sigma + 1}{2 \sigma}})}{(q + 1) \chi(q + 1)} - \frac{\exp(- |\psi(q) s|^{\frac{\sigma + 1}{2 \sigma}})}{q \chi(q)} > 0.
\end{equation}
If the inequality \eqref{eq:summandincreasing} holds for all $q < \frac{1}{|s|}$, then the summand is increasing in this domain, and is therefore maximized when $q = \frac{1}{|s|}$, establishing the bound \eqref{eq:sum2bound} for such $q$.

In order to establish \eqref{eq:summandincreasing}, it is enough to verify that the numerator of the difference is positive. This numerator is
\[\exp (- |\psi(q + 1) s|^{\frac{\sigma + 1}{2 \sigma}}) q \chi(q) - \exp( - |\psi(q) s|^{\frac{\sigma + 1}{2 \sigma}}) (q + 1) \chi(q + 1).\]
Since the logarithm is an increasing function, it is enough to show that 
\[-|\psi(q + 1) s|^{\frac{\sigma + 1}{2 \sigma}} + \log q + \log \chi(q) > - |\psi(q) s|^{\frac{\sigma + 1}{2 \sigma}} + \log (q + 1) + \log \chi(q + 1).\]
This inequality is equivalent to
\begin{equation}\label{eq:summandincreasing2}
\log (q + 1) - \log q + \log \chi(q + 1) - \log \chi(q) < |s|^{\frac{\sigma + 1}{2 \sigma}} (\psi(q)^{\frac{\sigma + 1}{2 \sigma}} - \psi(q + 1)^{\frac{\sigma + 1}{2 \sigma}}).
\end{equation}
The Taylor series for the logarithm guarantees that $\log (q + 1) - \log q = \frac{1}{q} + O \left(\frac{1}{q^2}\right)$; the subpolynomial growth condition \eqref{eq:chigrowth} guarantees that $\log \chi(q + 1) - \log \chi(q) = o \left(\frac{1}{q} \right)$. In total, the left side of inequality \eqref{eq:summandincreasing2} is $\frac{1}{q} + o \left(\frac{1}{q} \right)$. On the other hand, since we are in the regime where $q < \psi^{-1}(1/|s|)$, the right side of \eqref{eq:summandincreasing2} is bounded below by
\[|s|^{\frac{\sigma + 1}{2 \sigma}} (\psi(q)^{\frac{\sigma + 1}{2 \sigma}} - \psi(q + 1)^{\frac{\sigma + 1}{2 \sigma}}) \geq 1 - \left(\frac{\psi(q + 1)}{\psi(q)}\right)^{\frac{\sigma + 1}{2 \sigma}}.\]
By \eqref{eq:sigmadef}, we have
\[\left(\frac{\psi(q + 1)}{\psi(q)}\right)^{\frac{\sigma + 1}{2 \sigma}} \leq \left(\frac{q}{q + 1} \right)^{\frac{\sigma + 1}{2}}.\]
Hence,
\begin{IEEEeqnarray*} {rCl}
1 - \left(\frac{\psi(q + 1)}{\psi(q)}\right)^{\frac{\sigma + 1}{2 \sigma}} & \geq & 1 - \left(\frac{q}{q + 1} \right)^{\frac{\sigma + 1}{2}} \\
& = & \left(\frac{\sigma + 1}{2}\right) \frac{1}{q} + o \left(\frac{1}{q^2} \right).  
\end{IEEEeqnarray*}

Since $\frac{\sigma + 1}{2} > 1$, we see that the inequality \eqref{eq:summandincreasing2} holds for $\psi(1 / \sqrt{|s|}) \leq q \leq \psi(1 / |s|)$ provided that $M_k$ (and hence $|s|$) is sufficiently large.

Hence we have the estimate
\[\mathrm{II} \lesssim \theta(|s|).\]
\paragraph*{Estimate for $\mathrm{III}$.} For the final sum, we apply the estimate \eqref{eq:schwartztail} to $\hat \phi$ to get
\begin{IEEEeqnarray*} {rCl}
\sum_{\substack {q <\psi^{-1}(1/\sqrt{|s|}) \\ {\substack q \text{ prime}} \\ q \vert s}} \frac{\hat \phi\left(\psi(q) s\right)}{q\chi(q)} &\lesssim& \sum_{\substack {q < \psi^{-1}(1/ \sqrt{s}) \\ {\substack q \text{ prime}} \\ q \vert s}} \frac{\exp \left(-\left| \psi(q) s \right|^\frac{\sigma + 1}{2 \sigma}\right)}{q}  \\
& \leq & \sum_{\substack {q <\psi^{-1}(1/\sqrt{|s|}) \\ {\substack q \text{ prime}} \\ q \vert s}} \frac{\exp \left(-\left|s \right|^\frac{\sigma + 1}{4 \sigma}\right)}{q}\\
& \leq & \exp \left(-\left|s \right|^\frac{\sigma + 1}{4 \sigma}\right) \log \left(\psi^{-1}(1/\sqrt{|s|})\right)\\
& \lesssim & \theta(|s|).
\end{IEEEeqnarray*}

For the estimate \eqref{eq:gkhatslarge2.1}, we observe that $|s|$ is sufficiently large for the estimate \eqref{eq:schwartztail} to apply to $\hat \phi$ for every $q \in [M_k, \beta(M_k)]$. As such
\begin{IEEEeqnarray*} {rCl}
\left|\hat g_k(s)\right| & \lesssim &\sum_{\substack {M_k \leq q \leq \beta(M_k) \\ {\substack q \text{ prime}} \\ q \vert s}} \frac{\exp\left(-(\psi(q)|s|)^{\frac{\sigma +1}{2\sigma}}\right)}{{\chi(q)q}}\\
& \leq & \sum_{\substack {M_k \leq q \leq \beta(M_k) \\ {\substack q \text{ prime}} \\ q \vert s}} \frac{\exp\left(-(\psi(\beta(M_k))|s|)^{\frac{\sigma +1}{2\sigma}}\right)}{{\chi(M_k)M_k}}.\\
\end{IEEEeqnarray*}
The inequality $|s| \geq \psi(\beta(M_k))^{-2}$ gives us $\psi(\beta(M_k)) \leq \frac{1}{\sqrt{|s|}}$. Therefore,
\[\left|\hat g_k(s)\right|  \lesssim \sum_{\substack {M_k \leq q \leq \beta(M_k) \\ {\substack q \text{ prime}} \\ q \vert s}} \frac{\exp\left(-|s|^{\frac{\sigma +1}{4\sigma}}\right)}{{\chi(M_k)M_k}}.\]

Observe that the number of summands is less than $\beta(M_k)$. Moreover, we may disregard the denominator, for large $M_k$, to derive an upper bound. Hence,
\[\left|\hat g_k(s)\right| \lesssim \beta(M_k) \exp\left(-|s|^{\frac{\sigma +1}{4\sigma}}\right).\]

Now, we need to eliminate the $\beta(M_k)$ term from the estimate but this will be at the cost some decay from the exponent. Rewrite the above inequality as 
\begin{IEEEeqnarray*} {rCl}
\left|\hat g_k(s)\right| &\lesssim& \beta(M_k) \exp\left(-\frac{1}{2}|s|^{\frac{\sigma +1}{4\sigma}}\right)\exp\left(-\frac{1}{2}|s|^{\frac{\sigma +1}{4\sigma}}\right) \\
& \leq &  \beta(M_k) \exp\left(-\frac{1}{2}\psi(\beta(M_k))^{-\frac{\sigma +1}{2\sigma}}\right)\exp\left(-\frac{1}{2}|s|^{\frac{\sigma +1}{4\sigma}}\right) \IEEEyesnumber \label{eq:SplitExponent}
\end{IEEEeqnarray*}
when we apply $|s| \geq \psi(\beta(M_k))^{-2}$. From the equation \eqref{eq:taudef}, when $M_k$ is large enough we have
\[\frac{1}{2} \tau \leq -\frac{\log \psi(\beta(M_k))}{\log \beta(M_k)} \leq 2 \tau\]
which may be rewritten as
\[-\frac{1}{2} \tau \log \beta(M_k) \geq \log \psi(\beta(M_k)) \geq -2 \tau \log \beta(M_k).\]
Exponentiating gives
\begin{equation} \label{eq:PsiBetaInequality}
\beta(M_k)^{-\frac{1}{2} \tau} \geq \psi(\beta(M_k)) \geq \beta(M_k)^{-2 \tau}.
\end{equation}

Applying the upper bound from  the equation \eqref{eq:PsiBetaInequality} to \eqref{eq:SplitExponent}, we get
\[\left|\hat g_k(s)\right| \lesssim \beta(M_k) \exp\left(-\frac{1}{2}\beta(M_k)^{\frac{\tau(\sigma +1)}{4\sigma}}\right)\exp\left(-\frac{1}{2}|s|^{\frac{\sigma +1}{4\sigma}}\right).\]
For large $M_k$, we observe that the exponential term dependent on $M_k$ is decaying much faster than $\beta(M_k)$. Hence,
\begin{IEEEeqnarray*}{rCl}
\left|\hat g_k(s)\right| & \lesssim & \exp\left(-\frac{1}{2}|s|^{\frac{\sigma +1}{4\sigma}}\right) \\
& \lesssim & \exp \left(-\frac{1}{2} (\psi(\beta(M_k))^2 |s|)^{\frac{\sigma + 1}{4 \sigma}} \right).
\end{IEEEeqnarray*}
\end{proof}

\subsection{Single-factor estimate for Theorem \ref{thm:fastdecay}}\label{subsec:singlescalefastdecay}
In the case of Theorem \ref{thm:fastdecay}, it is more convenient to choose the function $g_k$ to be supported in a neighborhood of rational numbers with different denominators at very different scales. Thus, only one denominator will meaningfully contribute to the value of $|\hat g_k(s)|$.

As in subsection \ref{subsec:singlescaleslowdecay}, we begin by defining a smooth function $\phi$ with its support in the interval $[-1/2, 0]$ satisfying the conditions 
\begin{equation}\label{eq:phinormalization2}
\hat \phi(0) = 1
\end{equation}
and 
\begin{equation}\label{eq:schwartztail2}
\hat \phi(s) \lesssim \exp \left(-|s|^{3/4} \right).
\end{equation}

Let $n_k$ be an increasing sequence of integers to be specified later. For a given $k$, we choose $q_{k,1}, \ldots, q_{k,n_k}$ of prime numbers as follows. First, we choose $q_{k,1}$ to be a large prime number. We choose the remaining $q_{k,j}$ so that $q_{k,2} \gg \frac{1}{\Psi(q_{k,1})}$, $q_{k,3} \gg \frac{1}{\Psi(q_{k,2})}$,$\ldots$, $q_{k,n_k} \gg \frac{1}{\Psi(q_{k,n_{k-1}})}$. Furthermore, we also assume that for each $j$, we have 
\begin{equation}\label{eq:qkjcondition}
\max \left(\frac{1}{q_{k,j}}, \psi(q_{k,j}) \right) < \frac{1}{2} \psi(q_{k,j-1}).
\end{equation} 
Define
\[g_k(x) = \frac{1}{n_k} \sum_{j = 1}^{n_k} \frac{1}{q_{k,j} \psi(q_{k,j})}\sum_{p = 1}^{q_{k,j}}\phi\left(\frac{1}{\psi(q_{k,j})}\left(x-\frac{p}{q_{k,j}}\right)\right).\]
Then
\[\hat g_k(s) = \frac{1}{n_k} \sum_{j = 1}^{n_k} \frac{1}{q_{k,j}}\sum_{p = 1}^{q_{k,j}}\hat \phi(\psi(q_{k,j}) s) \e\left(\frac{p}{q_{k,j}}s\right).\]

Remove any terms for which $q_{k,j}$ does not divide $s$ to get
\begin{equation}\label{eq:gkhatsum}
\hat g_k(s) = \frac{1}{n_k} \sum_{\substack{1 \leq j \leq n_k\\ q_{k,j} \vert s}} \hat \phi(\psi(q_{k,j}) s) .  
\end{equation}

\begin{lemma}\label{lem:gkestimatefastdecay}
Suppose that $g_k$ is defined as above. Then we have the following estimates for $s \in \mathbb{Z}$.
\begin{IEEEeqnarray}{rCll}
\hat g_k(0) & = & 1 \label{eq:gkhat02.5} \\
\hat g_k(s) & = & 0 \quad & \text{if $0 < |s| < q_{k,1}$} \label{eq:gkhatssmall2.5} \\
|\hat g_k(s)| & \lesssim & \frac{1}{n_k} \quad & s \neq 0 \label{eq:gkeverywhere2.5}\\
|\hat g_k(s)| & \lesssim & \exp \left(-\frac{1}{2} |\psi(q_{k,n_k}) s|^{\frac{3}{4}} \right) \quad & \text{if $|s| \geq \psi(q_{k,n_k})^{-1}$.}\label{eq:gkhatslarge2.5}
\end{IEEEeqnarray}
\end{lemma}
\begin{proof}
First, it is clear from \eqref{eq:phinormalization2} and \eqref{eq:gkhatsum} that $\hat g_k(0) = 1$, establishing \eqref{eq:gkhat02.5}. Moreover, the sum \eqref{eq:gkhatsum} is seen to be empty if $0 < |s| < q_{k,1}$, establishing \eqref{eq:gkhatssmall2.5}.

To prove \eqref{eq:gkeverywhere2.5}, we split the sum \eqref{eq:gkhatsum} depending on the size of $q_{k,j}$ relative to $s$. Suppose $j_0(s)$ is such that $\psi(q_{k,j_0})|s| > 1$, but such that $\psi(q_{k,j_0 + 1})|s| \leq 1$, taking $j_0(s) = 0$ if $\psi(q_{k,1} )|s| < 1$ or $j_0 = n_k$ if $\psi(q_{k, n_k})|s| > 1$.
\[|\hat g_k(s)|  \leq  \frac{1}{n_k} \sum_{\substack{j_0(s) + 1 \leq j \leq n_k \\ q_{k,j} \vert s}} \left|\hat \phi(\psi(q_{k,j}) s)\right| + \frac{1}{n_k} \sum_{\substack{1 \leq j \leq j_0(s) \\ q_{k,j} \vert s}} \left|\hat \phi(\psi(q_{k,j}) s)\right|\]

For the second sum, we may apply \eqref{eq:schwartztail2}, the Schwartz tail for $\hat \phi$. Hence, using the assumption \eqref{eq:qkjcondition},
\begin{IEEEeqnarray*} {rCl}
\frac{1}{n_k} \sum_{\substack{1 \leq j \leq j_0(s) \\ q_{k,j} \mid s}} \left|\hat \phi(\psi(q_{k,j}) s) \right| & \lesssim & \frac{1}{n_k} \sum_{\substack{1 \leq j \leq j_0(s)}} \exp \left(-\left| \psi(q_{k,j})s\right|^\frac{3}{4}\right) \\
& \lesssim & \frac{1}{n_k} \sum_{1 \leq j \leq j_0(s)} \exp \left( -2^{\frac{3(j_0(s) - j)}{4}} \right) \\
& \lesssim & \frac{1}{n_k}.
\end{IEEEeqnarray*}

For the first sum, recall that $j_0$ is chosen so that $\psi(q_{k,j_0 + 1})s < 1$. Since $q_{k, j} \geq \frac{1}{\psi(q_{k,j_0 + 1})}$ for any $j \geq j_0 + 2$, it follows that $\frac{1}{q_{k,j}} s < 1$ for such $j$. This means that it is impossible for $q_{k,j}$ to divide $s$ for $j > j_0 + 1$. Hence, the only term that can contribute to the sum is the $j = j_0 + 1$ term. To control the contribution of this term, we simply apply the bound  
\[\left|\hat \phi(\psi(q_{k,j}) s)\right| \leq 1\]
to bound the second sum by a constant times $\frac{1}{n_k}$. Thus, for any integer $s \neq 0$, we have the bound
\[|\hat g_k(s)| \lesssim \frac{1}{n_k}.\]
It remains to show the bound \eqref{eq:gkhatslarge2.5}. For $s \geq \psi(q_{k,n_k})^{-1}$, we can in fact apply the Schwartz bound \eqref{eq:schwartztail2} for $\phi$ to every summand in \eqref{eq:gkhatsum}. Hence 
\begin{IEEEeqnarray*}{rCl}
|\hat g_k(s)| & \leq & \frac{1}{n_k} \sum_{\substack{1 \leq j \leq n_k \\ q_{k, j} \mid s}} |\hat \phi(\psi(q_{k,j}) s)| \\
& \lesssim & \frac{1}{n_k} \sum_{1 \leq j \leq n_k} \exp \left( - | \psi(q_{k,j}) s|^{\frac{3}{4}} \right) \\
& \lesssim & \frac{1}{n_k} \sum_{1 \leq j \leq n_k} \exp \left(- |2^{n_k - j} \psi(q_{k, n_k}) s|^{\frac{3}{4}} \right) \\
& \lesssim & \exp \left(- |\psi(q_{k, n_k}) s|^{\frac{3}{4}} \right).
\end{IEEEeqnarray*}
\end{proof}
\section{Stability and convergence of $\hat \mu_{\chi,\omega}$}
In order to prove Theorems \ref{thm:slowdecay}, \ref{thm:compact}, and \ref{thm:fastdecay}, we will piece together the functions $g_k$ provided in Section \ref{sec:singlescale} across multiple scales. Lemmas \ref{lem:convolutionstability} and \ref{lem:convolutionstability2} are used to show that the Fourier transforms $\hat g_k$ of the functions $g_k$ do not exhibit much interference. The construction proceeds slightly differently in the case of Theorem \ref{thm:fastdecay}, as this theorem does not prescribe a specific decay rate for $\hat \mu$. 
\subsection{Construction of $\mu$ for Theorem \ref{thm:slowdecay} and Theorem \ref{thm:compact}}
Let $\psi$ and $\chi$ be functions satisfying the assumptions \eqref{eq:taudef}, \eqref{eq:sigmadef}, \eqref{eq:chicondition}, and \eqref{eq:chigrowth}. Recall that in the case of Theorem \ref{thm:compact} that we take $\chi \equiv 1$, and we showed in Remark \ref{rmk:compacthmconditions} that $\psi$ satisfies assumptions \eqref{eq:taudef} and \eqref{eq:sigmadef}.
We begin by constructing a sequence of functions $(\mu_{\chi, \omega, k})_{k \in \mathbb{N}}$ where $\mu_{\chi, \omega, k}(x)$ is the product 

\[\mu_{\chi, \omega, k}(x) = \prod_{i=1}^{k} g_i(x).\]

For each $g_i$ we choose an associated $M_i$ such that the estimates in Lemma \ref{lemma: gkestimate2.1} apply. We further assume that the $M_{i}$'s are spaced sufficiently far apart to satisfy the conditions of Lemma \ref{lem:convolutionstability}. In particular, this implies that for each $i \geq 1$ we have
\begin{equation}\label{eq:Mkcondition}
M_{i+1} \geq \psi(\beta(M_i))^{-2}.
\end{equation}

Taking the Fourier transform of this sequence, we get the sequence $(\hat \mu_{\chi, \omega, k})_{k \in \mathbb{N}}$ where 
\[\hat \mu_{\chi, \omega, k} (s) = \hat g_1 * \cdots * \hat g_k (s).\] 

With this sequence of functions defined, the next objective is to show that the sequence is uniformly convergent and that the functions $\hat \mu_{\chi, \omega, i}$ satisfy a similar decay estimate (up to a constant) for all $i$. We begin with the latter:
\begin{lemma} \label{lemma: stability2.1}
For the sequence of functions $(\hat \mu_{\chi, \omega, k})_{k \in \mathbb{N}}$ defined above, we have the following statements for any integers $k, l$ with $k > l$:

\begin{IEEEeqnarray} {rCll}
|\hat \mu_{\chi, \omega, k}(0)| & \leq & \text{ } 2 \quad & \label{eq:mu_k 0 2.1}\\ 
|\hat \mu_{\chi, \omega, l}(s) -\hat \mu_{\chi, \omega, k}(s)| & \lesssim & \sum_{j = l +1}^{k} M_{j}^{-99} \quad & \text{when } 0 \leq |s| < M_l/4\label{eq:mu_k uniform small 2.1}\\
|\hat \mu_{\chi, \omega, k}(s)| & \lesssim & \text{ } \theta(|s|) \omega(|s|) \quad &  \text{for all $s \neq 0$} \label{eq:mu_k small2.1}\\
|\hat \mu_{\chi, \omega, k}(s)| & \lesssim & \,  \exp \left(-\frac{1}{2} (\psi(\beta(M_k))^{2}|s|)^{\frac{\sigma + 1}{4 \sigma}} \right) \quad & \text{if $|s| \geq \psi(\beta(M_k))^{-2}$.} \label{eq:mu_k large 2.1}
\end{IEEEeqnarray}
Note that since $\mu$ is a positive measure, \eqref{eq:mu_k 0 2.1} implies that $|\hat \mu_{\chi, \omega, k}(s)| \leq 2$ for all $s$.
\end{lemma}
\begin{proof}
We prove Lemma \ref{lemma: stability2.1} by induction and repeated application of Lemma \ref{lem:convolutionstability}. We begin with the basis by letting $k = 2$. Then $\hat \mu_2 = \hat g_1 * \hat g_2$. Apply Lemma \ref{lem:convolutionstability} with $H = \hat g_1$, $G= \hat g_2$, $N_1 = \psi(\beta(M_1))^{-2}$, $N_2 = M_2$ and $N_3 = \psi(\beta(M_2))^{-2}$. Then the estimates \eqref{eq:mu_k uniform small 2.1}, \eqref{eq:mu_k small2.1} and  \eqref{eq:mu_k large 2.1} immediately follow from \eqref{eq:convolve1}, \eqref{eq:convolve2} and \eqref{eq:convolve3}, respectively. The statement \eqref{eq:mu_k 0 2.1} can be shown by the following calculation:
\begin{IEEEeqnarray*} {rCl}
|\hat \mu_{\chi, \omega, 1}(0)| & \leq & |\hat g_1(0) - \hat g_1* \hat g_2 (0)| + |\hat g_1(0)|\\
& \leq & \mathcal{O}(M_2^{-99}) + 1\\
& \leq & 2
\end{IEEEeqnarray*}
where the last inequality holds for the choice of a sufficiently large $M_2$. Now assume that Lemma \ref{lemma: stability2.1} holds for $k$ up to $n-1$. Then for the case $k = n$, we make the choice $H = \hat \mu_{n-1} = \hat g_1*\cdots*\hat g_{n-1}$, $G= \hat g_n$, $N_1 = \psi(\beta(M_{n-1}))^{-2}$, $N_2 = M_n$ and $N_3 = \psi(\beta(M_n))^{-2}$. From the induction hypothesis and Lemma \ref{lem:convolutionstability}, the estimates \eqref{eq:mu_k small2.1} and \eqref{eq:mu_k large 2.1} immediately follow. For estimate \eqref{eq:mu_k uniform small 2.1}, assume $l < k$ and assume $|s| \leq M_l/4$. Then the triangle inequality gives:
\[|\hat \mu_{\chi, \omega, l}(s) -\hat \mu_{\chi, \omega, k}(s)| \leq |\hat \mu_{\chi, \omega, l}(s) - \hat \mu_{\chi, \omega, l+1}(s)| + \cdots + |\hat \mu_{\chi, \omega, k-1}(s) - \hat \mu_{\chi, \omega, k}(s)|.\]

By the induction hypothesis, $|\hat \mu_{i}(s) - \hat \mu_{i+1}(s)| \lesssim M_{i+1}^{-99}$ for $l \leq i \leq k-2$ and Lemma \ref{lem:convolutionstability} gives
\[|\hat \mu_{\chi, \omega, k-1}(s) - \hat \mu_{\chi, \omega, k}(s)|  \lesssim M_{k}^{-99}.\] 
Consequently,
\[|\hat \mu_{\chi, \omega, l}(s) -\hat \mu_{\chi, \omega, k}(s)| \lesssim \sum_{j = l+1}^{k} M_{j}^{-99}.\]

Finally, from the calculation
\begin{IEEEeqnarray*} {rCl}
|\hat \mu_{\chi, \omega, k}(0)| & \leq & |\hat \mu_{\chi, \omega, 1} (0) - \hat \mu_{\chi, \omega, k} (0)| + |\hat \mu_{\chi, \omega, 1} (0)|\\
& \leq & 1 +\mathcal{O}\left(\sum_{j = 2}^{k} M_{j}^{-99}\right)\\
& \leq & 2,
\end{IEEEeqnarray*}
the estimate \eqref{eq:mu_k 0 2.1} is proved. 
\end{proof}
Turning now to proving the uniform convergence of the sequence $(\hat \mu_{\chi, \omega, k})_{k \in \mathbb{N}}$, we have the following lemma.
\begin{lemma}\label{lem:uniformconvergenceslowdecay}
The sequence $(\hat \mu_{\chi, \omega, k})_{k \in \mathbb{N}}$ converges uniformly for all $s \in \mathbb{Z}$ to some function $M(s)$. This function $M(s)$ has the property that 
\begin{equation}\label{eq:weaklimit}
|M(s)| \lesssim \theta(|s|) \omega(|s|); \text{ }s \in \mathbb{Z}
\end{equation}
\end{lemma}
\begin{proof} Let $\epsilon > 0$. There exists an $m_0$, depending on $\epsilon$ and $\omega$, sufficiently large such that 
\[\theta(|s|) \omega(|s|) < \epsilon/2C\]
when $|s| \geq M_{m_0}/4$. Here $C$ is taken to be the implicit constant for estimate \eqref{eq:mu_k small2.1}. Then for $m\geq n\geq m_0$ 
\begin{IEEEeqnarray*} {rCl}
 |\hat \mu_{\chi, \omega, m}(s) - \hat \mu_{\chi, \omega, n}(s)| & \leq & |\hat \mu_{\chi, \omega, m} (s)| + |\hat \mu_{\chi, \omega, n} (s)|\\
 & < & \epsilon.
\end{IEEEeqnarray*}
When $0 \leq |s| \leq M_{m_0}/4$, applying the estimate \eqref{eq:mu_k uniform small 2.1} gives
\[ |\hat \mu_{\chi, \omega, m}(s) - \hat \mu_{\chi, \omega, n}(s)|  \leq  \sum_{j = n}^{\infty} M_{j}^{-99}\]
and the sum may be made to be less than $\epsilon$.

Hence the sequence $\hat \mu_{\chi, \omega, n}$ has a uniform limit $M(s)$. An upper bound on $|M(s)|$ will follow from Lemma \ref{lemma: stability2.1}. Suppose $|s|$ is such that $\frac{M_k}{4} \leq |s| \leq \frac{M_{k+1}}{4}$.

Then the estimate \eqref{eq:mu_k small2.1} gives that
\[|\hat \mu_{\chi, \omega, k}(s) | \lesssim \text{ } \theta(|s|) \omega(|s|),\]
and \eqref{eq:mu_k uniform small 2.1} and the triangle inequality give
\begin{IEEEeqnarray*}{rCl}
|M(s)| & \leq & |\hat \mu_{\chi, \omega, k}(s)| + \limsup_{l \geq k} |\hat \mu_{\chi, \omega, k}(s) - \hat \mu_{\chi, \omega, l}(s)| \\
& \lesssim & \theta(|s|) \omega(|s|) + \sum_{j=k+1}^{\infty} M_j^{-99} \\
& \lesssim & \theta(|s|) \omega(|s|) + |s|^{-99} \\
& \lesssim & \theta(|s|) \omega(|s|),
\end{IEEEeqnarray*}
as desired.
\end{proof}
In order to show that the sequence $\mu_{\chi, \omega, n}$ converges to a weak limit $\mu$ using the convergence of the $\hat \mu_{\chi, \omega, n}(s)$, it is normal to appeal to a theorem such as the L\'evy continuity theorem. However, this is slightly inconvenient as we only have estimates for $\hat \mu_{\chi, \omega, n}(s)$ at integer values $s$. We will provide a proof of the weak convergence below. First, we will need the following technical lemma relating the Fourier series of a measure supported on the interval $[0,1]$ to its Fourier transform. A stronger version of this lemma can be found as Lemma 1 of Chapter 17 in the book of  Kahane \cite{Kahane85}.
\begin{lemma}\label{lem:kahane}
Suppose that $\mu$ is a measure supported on the interval $[0,1]$ satisfying an estimate of the form
\[|\hat \mu(s)| \lesssim N(|s|) \quad \text{for all $s \in \mathbb{Z}$}\]
where $N: \mathbb{R}^+ \to \mathbb{R}^+$ is a non-increasing function satisfying the doubling property
\begin{equation}\label{eq:doubling}
N(\xi/2) \lesssim N(\xi) \quad \text{for all $\xi \in \mathbb{R}^+$}.
\end{equation}
Then $|\hat \mu(\xi)| \lesssim N(|\xi|)$ for all $\xi \in \mathbb{R}$.
\end{lemma}
We have already seen that $\theta(\xi) \omega(\xi)$ is a doubling function for $\xi > 0$. Thus we can apply Lemma \ref{lem:kahane}.
\begin{lemma}\label{lem:weaklimit}
The sequence of measures $\mu_{\chi, \omega, k}$ has a weak limit $\mu_{\chi, \omega}$. This weak limit $\mu_{\chi, \omega}$ satisfies the estimate
\begin{equation}\label{eq:muhatdecayreal}
\hat \mu_{\chi, \omega}(\xi) \lesssim \theta (|\xi|) \omega(|\xi|)
\end{equation}
for all real numbers $\xi$.
\end{lemma}
\begin{proof}
Observe that each measure $\mu_{\chi, \omega, n}$ has total variation norm bounded by $2$. We claim that the measures $\mu_{\chi, \omega, k}$ have a weak limit. First, by the Banach-Alaoglu theorem, there exists a subsequence $\mu_{\chi, \omega, n_k}$ that has a weak limit $\mu_{\chi, \omega}$.  Since each measure $\mu_{\chi, \omega, n_k}$ is supported in $[0,1]$, the weak limit $\mu_{\chi, \omega}$ is supported in $[0,1]$. 

In particular, since $\mu$ is supported in $[0,1]$, each Fourier coefficient $\hat \mu(s)$ of $\mu$ is obtained by integrating against a continuous, compactly supported function. Hence, for each $s \in \mathbb{Z}$, $\lim_{k \to \infty} \hat \mu_{\chi, \omega, n_k}(s) = M(s)$, where $M(s)$ is the limit in Lemma \ref{lemma: stability2.1}. 

By the corollary to Theorem 25.10 of Billingsley \cite{Billingsley12}, it is enough to check that each weakly convergent subsequence of $\{\mu_{\chi, \omega, n}\}$ converges weakly to $\mu_{\chi, \omega}$. Suppose $\{\nu_{\chi, \omega, n}\}$ is a subsequence of the $\mu_{\chi, \omega, n}$ with some weak limit $\nu$. Then $\nu$ is supported on $[0,1]$, so by the same argument as in the previous paragraph, Lemma \ref{lemma: stability2.1} implies that $\hat \nu(s) = M(s)$ for every $s \in \mathbb{Z}$. Since a measure supported on $[0,1]$ is uniquely determined by its Fourier-Stieltjes series, it follows that $\nu = \mu_{\chi, \omega}$ as desired.

Finally, we verify that $\hat \mu_{\chi, \omega}(\xi)$ satisfies the estimate \eqref{eq:muhatdecayreal}. This estimate holds for integer values of $s$ by the estimate \eqref{eq:weaklimit}. Hence, Lemma \ref{lem:kahane} shows that $\hat \mu_{\chi, \omega}(\xi)$ satisfies the same estimate for $\xi \in \mathbb{R}$.
\end{proof}
Hence the measures $\mu_{\chi, \omega, k}$ have a weak limit supported on $[0,1]$. We now verify that this weak limit is indeed supported on the set $E(\psi)$.
\begin{lemma}\label{lem:support}
Let $\mu$ be as in Lemma \ref{lem:weaklimit}. Then $\mu$ is supported on $E(\psi)$.
\end{lemma}
\begin{proof}
It is easy to see that
\[\supp \phi_{p,q} \subset \left[\frac{p}{q} - \frac{1}{2}\psi(q), \frac{p}{q} + \frac{1}{2}\psi(q)\right]\]
and therefore
\[\supp g_k \subset  \bigcup_{\substack{M_{k} \leq q \leq \beta(M_{k}) \\ q \text{ prime}}} \bigcup_{p = 0}^{q-1} \left[\frac{p}{q} - \frac{1}{2}\psi(q), \frac{p}{q} + \frac{1}{2}\psi(q)\right].\]
Since each $\mu_{\chi, \omega, k}$ is the product of $g_i$'s its support is an intersection of these supports.
\[\supp \mu_{\chi, \omega, k} \subset \bigcap_{i = 1}^{k}\bigcup_{\substack{M_{i} \leq q \leq \beta(M_{i}) \\ q \text{ prime}}} \bigcup_{p = 0}^{q-1} \left[\frac{p}{q} - \frac{1}{2}\psi(q), \frac{p}{q} + \frac{1}{2}\psi(q)\right].\]
Because the measure $\mu_{\chi, \omega}$ is defined as the weak limit of the measures $\mu_{\chi, \omega, k}$, we have the containment
\[\supp \mu_{\chi, \omega} \subset \bigcap_{i = 1}^{\infty}\bigcup_{\substack{M_{i} \leq q \leq \beta(M_{i}) \\ q \text{ prime}}} \bigcup_{p = 0}^{q-1} \left[\frac{p}{q} - \frac{1}{2}\psi(q), \frac{p}{q} + \frac{1}{2}\psi(q)\right].\]
Observe that if $x \in \supp \mu_{\chi, \omega}$ and $k \in \mathbb{N}$, then $x$ must also lie in one of the intervals 
\[\left[\frac{p}{q} - \frac{1}{2}\psi(q), \frac{p}{q} + \frac{1}{2}\psi(q)\right]\]
for some $M_k \leq q \leq \beta(M_k).$

Therefore, there exists an infinite number of rational numbers $\frac{p}{q}$ which satisfy
\[\left| x-\frac{p}{q}\right| \leq \psi (q)\]
and we may conclude that $\supp \mu \subset E(\psi)$.
\end{proof}
The measure $\mu_{\chi, \omega}$ satisfies all of the properties required to prove Theorem \ref{thm:slowdecay}. Hence, the proof of Theorem \ref{thm:slowdecay} is complete.

To show Theorem \ref{thm:compact}, it is also necessary to verify that the support of $\mu$ is contained in a set of generalized $\alpha$-Hausdorff measure equal to zero. This will be shown in Section \ref{sec:hausdorffmeasure}.
\subsection{Construction of $\mu$ for Theorem \ref{thm:fastdecay}}
We now construct the measure $\mu$ described in Theorem \ref{thm:fastdecay}. The biggest difference between this construction and the one in the previous subsection is that we do not state explicit quantitative estimates describing the decay of the Fourier transform of the measures.

Choose a positive integer $n_1$ and let $M_1$ be a large integer. We will choose the sequences $\{n_j : j \geq 2\}$ and $\{M_j : j \geq 2\}$ to be rapidly increasing sequences of integers satisfying a certain set of conditions below. For each $j$, we choose prime numbers  $q_{j, 1}, \ldots, q_{j,n_j}$ with $M_j \leq q_{j,1} \ll \cdots \ll q_{j, n_j}$. When we choose the $M_j$, we will impose the condition that $M_{j+1} \gg q_{j, n_j}$ as well. Given $q_{j,1}, \ldots, q_{j,n_j}$ we define the function $g_j$ as in Subsection \ref{subsec:singlescalefastdecay}.

We define the function $\mu_k$ to be the pointwise product
\[\mu_k(x) = \prod_{j=1}^k g_j(x)\]
so $\hat \mu_1(s) = \hat g_1(s)$ and so that for any $k \geq 2$
\[\hat \mu_k(s) = \hat g_k(s) * \hat \mu_{k-1}(s).\]

We are now ready to state the main estimate on $\hat \mu_k$.
\begin{lemma}[Main estimate for $\hat \mu_k$]\label{lem:mukfastdecaymainest}
Suppose that the functions $g_k$ are chosen as above. Then provided that the sequences $n_j$ and $M_j$ are chosen appropriately, the measures $\hat \mu_k$ satisfy the following estimates for all integers $k \geq l$. All implicit constants below are assumed to be independent of $k$ and $l$.
\begin{IEEEeqnarray}{rCll}
|\hat \mu_k(0)| & \leq & 2 \label{eq:muk0fastedecay} \\
|\hat \mu_k(s) -  \hat \mu_l(s)| & \leq & \sum_{j=l+1}^{k} M_j^{-99} \quad & \text{if $0 \leq |s| \leq M_l/4$} \label{eq:muksmallfastdecay} \\
|\hat \mu_k(s)| & \lesssim &  n_k^{-1/2} \quad & \text{if $|s| \geq M_k/4$} \label{eq:mukmediumfastdecay}\\
|\hat \mu_k(s)| & \lesssim & \exp \left(-\frac{1}{2} \left|\frac{1}{8} \psi(q_{k, n_k}) s \right|^{\frac{3}{4}} \right) \quad & \text{if $|s| \geq 8\psi(q_{k,n_k})^{-1}$}\label{eq:muklargefastdecay}.
\end{IEEEeqnarray}
\end{lemma}
\begin{proof}
Let $n_1$ and $M_1$ be positive integers, and choose prime numbers $q_{1, 1}, \ldots, q_{1, n_1}$ such that $1 \leq q_{1,1} < q_{1,2} < \cdots < q_{1, n_1}$ satisfy the conditions of Lemma \ref{lem:gkestimatefastdecay}. Then $\hat g_1$ satisfies the estimates of Lemma \ref{lem:gkestimatefastdecay} and in particular satisfies the estimates of Lemma \ref{lem:mukfastdecaymainest}.

Given $g_1, \ldots, g_k$ such that $\mu_k$ satisfies the four conditions above, we will describe how to choose the integers $n_{k+1}$ and $M_{k+1}$ and how to choose the function $g_{k+1}$ so that $\mu_{k+1}$ will satisfy the four conditions above. Let $N_1 = \psi(\beta(M_k))^{-1}$. Lemma \ref{lem:convolutionstability2} requires that the quantity $\delta$ is chosen so that $\delta < \frac{1}{N_1^2}$; hence, we select $n_{k+1} = 100 N_1^2$. Choose $M_{k+1} = N_2 \gg n_k$ to be a prime number that is sufficiently large to satisfy the conditions of Lemma \ref{lem:convolutionstability2}. Take $N_2 = q_{k+1, 1} < \cdots < q_{k+1, n_{k+1}}$ sufficiently well-spaced to satisfy the conditions of Lemma \ref{lem:gkestimatefastdecay}. Then, choose $N_3 = \frac{1}{\psi(q_{k+1, n_{k+1}})} \gg q_{k+1,1}$. With these choices, we define $g_{k+1}$ as in Subsection \ref{subsec:singlescalefastdecay}. Hence Lemma \ref{lem:gkestimatefastdecay} implies that $\hat g_{k+1}$ satisfies the estimates required to serve as the function $G$ in Lemma \ref{lem:convolutionstability2}. 

Hence, we can apply Lemma \ref{lem:convolutionstability2} with $H = \hat \mu_k, G = \hat g_{k+1}$, $N_1 = \psi(q_{k, n_k})^{-1}$, $N_2 = q_{k+1,1}$, and $N_3 = \psi(q_{k+1, n_{k+1}})^{-1}$, and $\delta = \frac{1}{n_{k+1}}$.

This implies the estimates
\begin{IEEEeqnarray}{rCll}
| \hat \mu_{k+1}(s) - \hat \mu_{k}(s)| & \leq &  M_{k+1}^{-99} \quad & \text{if $0 \leq |s| \leq \frac{M_{k+1}}{4}$} \label{eq:mukplus1smallfastdecay}\\
|\hat \mu_{k+1}(s)| & \lesssim & n_{k+1}^{-1/2} \quad & \text{if $|s| \geq \frac{M_{k+1}}{4}$} \label{eq:mukplus1mediumfastdecay}\\
|\hat \mu_{k+1}(s)| & \lesssim & \exp \left(-\frac{1}{2} \left|\frac{1}{8} \psi(q_{k+1, n_{k+1}}) s \right|^{\frac{3}{4}}  \right) \quad & \text{if $|s| \geq 8\psi(q_{k+1,n_{k+1}})^{-1}$.}\label{eq:mukplus1largefastdecay}
\end{IEEEeqnarray}
Hence $\hat \mu_{k+1}$ satisfies the estimates \eqref{eq:mukmediumfastdecay} and \eqref{eq:muklargefastdecay}. In order to check \eqref{eq:muksmallfastdecay}, assume $l < k+1$ and $|s| \leq \frac{M_l}{4}$. If $l = k$, then the inequality follows from \eqref{eq:mukplus1smallfastdecay}. If $l < k$, then applying the inductive assumption \eqref{eq:muksmallfastdecay} to estimate the difference $\hat \mu_k - \hat \mu_l$ and applying \eqref{eq:mukplus1smallfastdecay} to estimate the difference $\hat \mu_{k+1}- \hat \mu_k$ gives
\begin{IEEEeqnarray*}{rCl}
|\hat \mu_{k+1}(s) - \hat \mu_l(s)| & \leq & |\hat \mu_{k+1}(s) - \hat \mu_k(s)| + |\hat \mu_k(s) - \hat \mu_l(s)| \\
& \leq & M_{k + 1}^{-99} + \sum_{j=l+1}^k M_k^{-99} \\
& = & \sum_{j=l+1}^{k+1} M_k^{-99}.
\end{IEEEeqnarray*}
This establishes \eqref{eq:muksmallfastdecay} for $\hat \mu_{k+1}$. Applying \eqref{eq:muksmallfastdecay} with $l = 1$ and $s = 0$, we see that
\begin{IEEEeqnarray*}{rCl}
|\hat \mu_{k+1}(0)| & \leq & |\hat \mu_{k+1}(0) - \hat \mu_1(0)| +  |\hat \mu_1(0)| \\
& \leq & \sum_{j=2}^{k+1} M_j^{-99} + 1 \\
& \leq & 2
\end{IEEEeqnarray*}
assuming the $M_j$ grow sufficiently rapidly.
\end{proof}
\begin{lemma}\label{lem:uniformconvergencefastdecay}
The sequence $\hat \mu_k$ converges uniformly for all $s \in \mathbb{Z}$ to a function $M(s)$. This function $M(s)$ has the property that $|M(s)| \to 0$ as $|s| \to \infty$.
\end{lemma}
\begin{proof}
The proof is similar to that of Lemma \ref{lem:uniformconvergenceslowdecay}. Let $\epsilon > 0$. Because $n_k \to \infty$, there is an index $k_0$ such that $n_{k_0}^{-1/2} + \sum_{j=k_0 + 1}^{\infty} M_j^{-99} < \epsilon / 2C$, where $C$ is the implicit constant from Lemma \ref{lem:mukfastdecaymainest}. 

Suppose $|s| > \frac{M_{k_0}}{4}$, and choose $l \geq k_0$ such that $\frac{M_{l}}{4} \leq |s| < \frac{M_{l + 1}}{4}$.  We have that $|\hat \mu_l(s)| \lesssim n_l^{-1/2} \leq n_{k_0}^{-\frac{1}{2}} < \frac{\epsilon}{2}$. Hence for $k \geq l$, we have
\begin{IEEEeqnarray*}{rCl}
|\hat \mu_k(s)| & \leq & |\hat \mu_l(s)| + |\hat \mu_k(s) - \hat \mu_l(s)| \\
& \leq & \frac{\epsilon}{2} + \sum_{j=l+1}^k M_j^{-99} \\
& \leq & \frac{\epsilon}{2} + \sum_{j=k_0 + 1}^{\infty} M_j^{-99} \\
& \leq & \epsilon.
\end{IEEEeqnarray*}
Hence $|\hat \mu_k(s)| \leq \epsilon$ for all $|s| \geq \frac{M_{k_0}}{4}$ and all $k \geq k_0$. 

If $|s| \leq \frac{M_{k_0}}{4}$ and $k_0 \leq l \leq k$, then we have
\[|\hat \mu_k(s) - \hat \mu_l(s)| \leq \sum_{j=l+1}^k M_j^{-99} \leq \sum_{j=k_0 + 1}^{\infty} M_j^{-99} < \epsilon/2.\] 
This proves that the sequence $\hat \mu_k(s)$ is uniformly Cauchy and hence uniformly convergent. Let $M(s)$ denote the uniform limit of this sequence.

Finally, we verify that $M(s) \to 0$ as $|s| \to \infty$. Suppose $|s|$ is such that $\frac{M_k}{4} \leq |s| \leq \frac{M_{k+1}}{4}$. Then we have from Lemma \ref{lem:mukfastdecaymainest} that $|\hat \mu_k(s)| \lesssim n_k^{-1/2}$, and
\begin{IEEEeqnarray*}{rCl}
M(s) & \lesssim & |\hat \mu_k(s)| + \limsup_{l \to \infty} |\hat \mu_l(s) - \hat \mu_k(s)| \\
& \lesssim & n_k^{-1/2} + \sum_{j=k+1}^{\infty} M_j^{-99} \\
& \lesssim & n_k^{-1/2}
\end{IEEEeqnarray*}
since $M_{k+1} \gg n_k$.  Since the sequence $n_k \to \infty$, this shows that $M(s) \to 0$ as $|s| \to \infty$, as desired.
\end{proof}

The rest of this proof is similar to the proof of Theorem \ref{thm:slowdecay}. In order to apply Lemma \ref{lem:kahane}, we use the fact from Lemma \ref{lem:doublingmajorant} that $M$ is majorized by $N(|s|)$, where $N$ is a doubling function. This will allow us to apply Lemma \ref{lem:kahane}.

We are now ready to show that the measures $\mu_k$ converge to a weak limit.

\begin{lemma}\label{lem:weaklimitfastdecay}
The sequence of measures $\mu_k$ has a weak limit $\mu$. This weak limit $\mu$ satisfies the estimate
\[|\hat \mu(\xi)| \to 0 \quad \text{as $|\xi| \to \infty$ in $\mathbb{R}$}.\]
Hence $\mu$ is a Rajchman measure.
\end{lemma}
\begin{proof}
This proof is almost exactly the same as the proof of Lemma \ref{lem:weaklimit}, but when we apply Lemma \ref{lem:kahane}, we use $N(|s|)$ as the bound on $M(s)$, where $N(s)$ is the function constructed in Lemma \ref{lem:doublingmajorant}.
\end{proof}

\begin{lemma}\label{lem:supportfastdecay}
Let $\mu$ be the weak limit in Lemma \ref{lem:weaklimitfastdecay}. Then the support of $\mu$ is contained in $E(\psi)$.
\end{lemma}
\begin{proof}
This lemma can be shown in a similar manner to Lemma \ref{lem:support}; we see that if $x \in \supp(\mu)$ then for each $k$, there exists a denominator $q_{k,j_k}$ and a numerator $p_{k, j_k}$ such that $\left|x - \frac{p_{k, j_k}}{q_{k, j_k}} \right| \leq \psi(q_{k, j_k})$; hence, $x$ is $\psi$-well-approximable.
\end{proof}
Therefore, the measure $\mu$ satisfies all of the properties promised in the statement of Theorem \ref{thm:fastdecay}. Thus the proof of Theorem \ref{thm:fastdecay} is complete.
\section{A bound on the generalized Hausdorff measure} \label{sec:hausdorffmeasure}
To complete the proof of Theorem \ref{thm:compact}, we must show that $F_\alpha$, which is taken to be the support of $\mu_{k,\omega}$, has zero $\alpha$-Hausdorff measure.
\begin{lemma}
Let $F_\alpha$ be a closed subset of
\[\left\{x : \left| x - \frac{r}{q} \right| < \psi(q) \, \text{for some integers $0 \leq r \leq q-1$, $M_k \leq q \leq \beta(M_k)$},\text{ } q \text{ prime}, 
 \text{ } k \in \mathbb{N}\right\}.\]
Let $\epsilon > 0$. Then there exists a cover $\mathcal{U}$ of $F_\alpha$ by open intervals $U$ such that
\[\sum_{U \in \mathcal{U}} \alpha(U) < \epsilon.\] 
\end{lemma}
\begin{proof}
The set $F_{\alpha}$ satisfies the following containment:
\[F_\alpha \subset \bigcap_{k=1}^{\infty} \bigcup_{\substack{M_k \leq q \leq \beta(M_k)\\ q \text{ prime}}} \bigcup_{p=0}^{q-1} \left[\frac{p}{q}-\psi(q),  \frac{p}{q}+\psi(q)\right].\]
For any $k$ the following collection of closed intervals is a cover for $F_\alpha$:
\[\left\{\left[\frac{p}{q}-\psi(q),  \frac{p}{q}+\psi(q)\right]: M_k \leq q \leq \beta(M_k), q \, \text{prime}, 0 \leq p \leq q-1\right\}.\]
Denote this collection as $\mathcal{I}_k$. The following collection $\mathcal{U}$ is also a cover of $F_\alpha$:
\[\mathcal{U} = \left\{J\bigcap\left[\frac{p}{q}-\psi(q),  \frac{p}{q}+\psi(q)\right]: M_k \leq q \leq \beta(M_k), q \, \text{prime}, 0 \leq p \leq q-1; J \in \mathcal{I}_{k-1}\right\}.\]
Fix a prime number $q$ with $M_k \leq q \leq \beta(M_k)$ and let $J \in \mathcal{I}_{k-1}$. Observe that the intersection of $J$ with the interval $\left[\frac{p}{q} - \psi(q), \frac{p}{q} + \psi(q) \right]$ is either empty or is a closed interval of length at most $2 \psi(q)$. We claim that the number of such intervals that intersect $J$ satisfies
\[\# \left\{p: \left[\frac{p}{q}-\psi(q),  \frac{p}{q}+\psi(q)\right]\cap J \neq \emptyset\right\} \sim |J|q,\]
where $|J|$ denotes the length of the interval $J$.

The interval $J$ belongs to $\mathcal{I}_{k-1}$. Therefore, $|J| \geq \psi(\beta(M_{k-1}))^{-2}$ by the assumption \eqref{eq:Mkcondition}. Hence $|J| \gg \frac{1}{M_k}$, and therefore, $|J| \gg \frac{1}{q}$.

The interval $[p/q - \psi(q), p/q + \psi(q)]$ intersects $J$ if and only if $p/q$ lies in a $\psi(q)$-neighborhood of $J$. Since $\psi(q) \approx q^{-\tau}$ by \eqref{eq:taudef} and $\tau > 2$, we have that $\psi(q) \ll 1/q$ if $k$ is sufficiently large. Hence, $[p/q - \psi(q), p/q + \psi(q)]$ intersects $J$ if and only if $p/q$ lies in an interval $J'$ of length $|J'| = |J| + 2 \psi(q) \sim |J|$.

Write $J' = [a,b]$. Then the smallest multiple of $p/q$ contained in $J'$ is $\frac{\lceil q a \rceil}{q}$, and the largest multiple of $p/q$ contained in $J'$ is $\frac{\lfloor q b\rfloor}{q}$. So the total number of multiples of $p/q$ contained in $J'$ is 
\begin{IEEEeqnarray*}{Cl}
& \lfloor qb \rfloor - \lceil q a \rceil + 1 \\
= & qb - qa + O(1) \\
= & q|J'| + O(1) \\
\sim &  q |J| + O(1).
\end{IEEEeqnarray*}
Since $|J| \gg 1/q$, we have $q |J| \gg 1.$ Therefore,
\[\#\left\{p : \left[\frac{p}{q} - \psi(q), \frac{p}{q} + \psi(q) \right] \cap J \neq \emptyset \right\} \sim |J|q,\]
as claimed.

Then
\begin{IEEEeqnarray*} {rCl}
\sum_{\mathrm{U} \in \mathcal{U}} \alpha(\diam(\mathrm{U})) & \leq & \sum_{J\in \mathcal{I}_{k-1}} \sum_{\substack{M_k \leq q \leq \beta(M_k)\\ q \text{ prime}}} \sum_{0\leq p \leq q-1} \alpha\left(\diam\left(J\bigcap\left[\frac{p}{q}-\psi(q),  \frac{p}{q}+\psi(q)\right]\right)\right)\\
& \sim& \sum_{J\in \mathcal{I}_{k-1}} |J| \sum_{\substack{M_k \leq q \leq \beta(M_k)\\ q \text{ prime}}} q \alpha(\psi(q)).    
\end{IEEEeqnarray*}
From assumption \eqref{eq:alpha def}, $\alpha(\psi(q)) = q^{-2}$. Therefore,
\[\sum_{J \in \mathcal{I}_{k-1}} |J| \sum_{\substack{M_k \leq q \leq \beta(M_k)\\ q \text{ prime}}} q \alpha(\psi(q)) = \sum_{J\in \mathcal{I}_{k-1}} |J| \sum_{\substack{M_k \leq q \leq \beta(M_k)\\ q \text{ prime}}} q^{-1}.\]
By choosing $\beta(M_k) = M_k^{\gamma}$ where $\gamma > 1$ is some positive number we get
\[\sum_{\substack{M_k \leq q \leq \beta(M_k)\\ q \text{ prime}}} q^{-1} \sim \log \log{M_k^{\gamma}} -  \log \log M_k = \log \gamma.\]
Consequently,
\begin{IEEEeqnarray*} {rCl}
\sum_{J\in \mathcal{I}_{k-1}} |J| \sum_{\substack{M_k \leq q \leq \beta(M_k)\\ q \text{ prime}}} q^{-1} &\sim& \sum_{J \in \mathcal{I}_k} |J|\\
& \lesssim & \sum_{\substack{M_{k-1} \leq q \leq \beta(M_{k-1})\\ q \text{ prime}}} \sum_{0 \leq p \leq q-1} \psi(q)\\
& = & \sum_{\substack{M_{k-1} \leq q \leq \beta(M_{k-1})\\ q \text{ prime}}} q \psi(q).
\end{IEEEeqnarray*}
Recall that $\psi(q) \approx q^{-\tau}$, so
\begin{IEEEeqnarray*} {rCl}
\sum_{\substack{M_{k-1} \leq q \leq \beta(M_{k-1})\\ q \text{ prime}}} q \psi(q) & \approx & \sum_{\substack{M_{k-1} \leq q \leq \beta(M_{k-1})\\ q \text{ prime}}} q^{-\tau+1}\\
& \lesssim & M_{k-1}^{-\tau+2}.
\end{IEEEeqnarray*}
The exponent $-\tau + 2 < 0$. Hence, if $k$ is chosen to be sufficiently large, we have 
\[\sum_{U \in \mathcal{U}} \alpha(U) < \epsilon,\]
as desired.
\end{proof}
\bibliographystyle{myplain}
\bibliography{Sharp_Fourier_Decay}
\end{document}